\documentclass[reqno]{amsart}

\usepackage[colorlinks=true, pdfstartview=FitV, linkcolor=blue, citecolor=blue, urlcolor=blue]{hyperref}

\usepackage{mathtools,amsmath, amsfonts, amssymb, latexsym, amsthm,
  amscd, cancel, enumerate, rotating, comment, appendix, ulem,
  makecell, paralist,times,graphicx,pgf,tikz,tikz-cd, times,algorithm,
  algpseudocode,dirtytalk}

\usepackage[margin=1.4in,letterpaper]{geometry}     
\usepackage[capitalise, noabbrev]{cleveref}
\usepackage[all]{xy}
\usepackage[latin1]{inputenc}
\usepackage[english]{babel}
\usepackage[mathscr]{eucal}
\usepackage{xxcolor}

\usepackage{caption}
\usepackage{subcaption}
\usepackage[T1]{fontenc}      

\usetikzlibrary{shapes}
\usetikzlibrary{snakes}
\usetikzlibrary{arrows}
\usepackage{tikz-cd}
\usetikzlibrary{matrix, decorations.pathmorphing}
\usetikzlibrary{decorations.markings}
\usetikzlibrary{calc}

\newif\ifdviwin

\numberwithin{equation}{section}

\theoremstyle{plain}
\newtheorem{theorem}{Theorem}[section]

\newtheorem*{Main Theorem}{Main Theorem}

\newtheorem{proposition}[theorem]{Proposition}
\newtheorem{lemma}[theorem]{Lemma}
\newtheorem{corollary}[theorem]{Corollary}

{\Alph{theorem}}
{\Alph{theorem}}

\theoremstyle{definition}
\newtheorem{notation}[theorem]{Notation}

\newtheorem{remark}[theorem]{Remark}
\newtheorem{definition}[theorem]{Definition}
\newtheorem{example}[theorem]{Example}

\newtheorem{question}[theorem]{Question}

  \newcounter{numlist} %
  {\end{list}}%

\theoremstyle{remark}
\newtheorem*{Claim}{Claim }
\newtheorem*{Claim1}{Claim 1}
\newtheorem*{Claim2}{Claim 2}

\newtheorem*{case1}{Case 1}
\newtheorem*{case2}{Case 2}
\newtheorem*{case3}{Case 3}

\newtheorem{chunk}[theorem]{}

\numberwithin{equation}{theorem}

\newcommand{\pd}{\mathrm{pd}}
\newcommand{\lcm}{\mathrm{lcm}}
\newcommand{\LCM}{\mathrm{LCM}}
\newcommand{\Tor}{\mathrm{Tor}}
\newcommand{\im}{\mathrm{im}}
\newcommand{\Taylor}{\mathrm{Taylor}}

\newcommand{\Scarf}{\mathrm{Scarf}}

\newcommand{\Supp}{\mathrm{Supp}}

\newcommand{\ba}{\mathbf{a}}
\newcommand{\bb}{\mathbf{b}}
\newcommand{\bc}{\mathbf{c}}
\newcommand{\bd}{\mathbf{d}}
\newcommand{\bm}{\mathbf{m}}
\newcommand{\f}{\mathbf{f}}
\newcommand{\m}{\mathbf{m}}

\newcommand{\bma}{\m^\ba}

\newcommand{\sfk}{\mathsf k}

\newcommand{\A}{\mathcal{A}}
\newcommand{\C}{\mathcal{C}}
\newcommand{\E}{\mathcal{E}}

\newcommand{\M}{\mathcal{M}}
\newcommand{\N}{\mathcal{N}}
\newcommand{\U}{\mathcal{U}}
\newcommand{\MS}{\mathcal{S}}
\newcommand{\J}{\mathcal{J}}
\newcommand{\X}{\mathcal{X}}

\newcommand{\LL}{\mathbb{L}}
\newcommand{\NN}{\mathbb{N}}
\newcommand{\RR}{\mathbb{R}}
\newcommand{\ZZ}{\mathbb{Z}}
\newcommand{\TT}{\mathbb{T}}
\renewcommand{\SS}{\mathbb{S}}
\newcommand{\UU}{\mathbb{U}}

\newcommand{\Ta}{\TT_\ba}

\newcommand{\Lrq}{\LL^r_q}
\newcommand{\UUrq}{\UU^r_q}

\newcommand{\Erq}{{\E_q}^r}
\newcommand{\Urq}{\U^r_q}
\newcommand{\Srq}{\SS^r_q}  
\newcommand{\Trq}{\TT^r_q}  
\newcommand{\Nrq}{\N^r_q}

\newcommand{\Ua}{\pmea \U_q^{r-|\ba|}}
\newcommand{\Ub}{\pmeb \U_q^{r-|\bb|}}
\newcommand{\Uc}{\pmec \U_q^{r-|\bc|}}

\renewcommand{\Ua}{\U_{\ba}}
\renewcommand{\Ub}{\U_{\bb}}
\renewcommand{\Uc}{\U_{\bc}}

\newcommand{\SE}{{S_\E}}

\newcommand{\e}{\epsilon}
\newcommand{\pme}{{\pmb{\e}}}
\newcommand{\pmea}{\pme^{\ba}}
\newcommand{\pmeb}{\pme^{\bb}}
\newcommand{\pmec}{\pme^{\bc}}
\newcommand{\pmed}{\pme^{\bd}}
\newcommand{\pmeaa}{\pme^{{\ba}'}}
\newcommand{\pmebb}{\pme^{{\bb}'}}
\newcommand{\pmecc}{\pme^{{\bc}'}}

\newcommand{\egcd}{\e\text{-}\mathrm{gcd}}

\newcommand{\qand}{\quad \mbox{and} \quad }
\newcommand{\qor}{\quad \mbox{or} \quad }
\newcommand{\qif}{\quad \mbox{if} \quad }
\newcommand{\qfor}{\quad \mbox{for} \quad }
\newcommand{\qwhere}{\quad \mbox{where} \quad }
\newcommand{\qwith}{\quad \mbox{with} \quad }
\newcommand{\qwhen}{\quad \mbox{when} \quad }
\newcommand{\qforeach}{\quad \mbox{for each} }
\newcommand{\qforsome}{\quad \mbox{for some} \quad }
\newcommand{\qforall}{\quad \mbox{for all} \quad }

\newcommand{\sm}{\setminus}
\newcommand{\ssm}{\smallsetminus}

\newcommand{\st}{\colon}

\begin{document}
\bibliographystyle{amsplain}

\author[S.~El Khoury]{Sabine El Khoury}
\address{Department of Mathematics,
American University of Beirut,
Bliss Hall 315, P.O. Box 11-0236,  Beirut 1107-2020,
Lebanon}
\email{se24@aub.edu.lb}

\author[S.~Faridi]{Sara Faridi}
\address{Department of Mathematics \& Statistics\\
Dalhousie University\\
6316 Coburg Rd.\\
PO BOX 15000\\
Halifax, NS\\
Canada B3H 4R2}
\email{faridi@dal.ca}

\author[L.~M.~\c{S}ega]{Liana M.~\c{S}ega}
\address{Liana M.~\c{S}ega\\ Division of Computing, Analytics and Mathematics\\School of Science and Engineering\\
University of Missouri-Kansas City\\ MO 64110\\ U.S.A.}
\email{segal@umkc.edu}

\author[S.~Spiroff]{Sandra Spiroff }
\address{Department of Mathematics,
University of Mississippi,
Hume Hall 335, P.O. Box 1848, University, MS 38677
USA}
\email{spiroff@olemiss.edu}



\title[Powers of extremal ideals]{The Scarf complex and betti numbers of \\
powers of extremal ideals}

\begin{abstract} 
  This paper is concerned with finding bounds on betti numbers and
  describing combinatorially and topologically (minimal) free
  resolutions of powers of ideals generated by a fixed number $q$ of
  square-free monomials. Among such ideals, we focus on a specific
  ideal $\E_q$, which we call {\it extremal}, and which has the
  property that for each $r\ge 1$ the betti numbers of $\Erq$ are an
  upper bound for the betti numbers of $I^r$ for any ideal $I$
  generated by $q$ square-free monomials (in any number of variables).
  We study the Scarf complex of the ideals $\Erq$ and use this
  simplicial complex to extract information on minimal free
  resolutions.  In particular, we show that $\Erq$ has a minimal free
  resolution supported on its Scarf complex when $q\le 4$ or when
  $r\le 2$, and we describe explicitly this complex. For any $q$ and
  $r$, we also show that $\beta_1(\Erq)$ is the smallest possible, or
  in other words equal to the number of edges in the Scarf complex.
  These results lead to effective bounds on the betti numbers of
  $I^r$, with $I$ as above.  For example, we obtain that $\pd(I^r)\le
  5$ for all ideals $I$ generated by $4$ square-free
  monomials and any $r\ge 1$.
\end{abstract} 
 
 \maketitle

\section{{\bf \large  Introduction}}

Understanding and bounding free resolutions of ideals is a central
problem in commutative algebra, and one that has led to the creation
of a broad set of tools, many in combinatorics, to better understand
them. Powers of a fixed ideal and their resolutions, though indispensable in
many areas of commutative algebra and algebraic geometry, are much
less understood.
 
Extremal ideals, introduced in~\cite{Lr}, were designed so that their
powers would have the largest minimal free resolutions among powers of
all square-free monomial ideals. As a result, many of the homological
invariants of their powers provide effective bounds for the same
invariants of powers of any square-free monomial ideal.

More precisely, an extremal ideal $\E_q$ has $q$ square-free monomial
generators, and if $I$ is {\it any} ideal in a polynomial ring, 
minimally generated by $q$ square-free monomials, then we proved, with
our coauthors in~\cite[Theorem 7.9]{Lr}, that 
\begin{equation}\label{betti}
\beta_i(I^r) \le \beta_i(\Erq)
\qforall r\geq 1.
\end{equation}
Thus, the problem of finding an effective upper bound for the betti
numbers of $I^r$ can be reduced to finding the betti numbers of the
ideals $\Erq$. However, despite the combinatorial construction of
extremal ideals, understanding the structure of the minimal free
resolution of their powers has proved to be a challenging problem.
This paper takes the first steps in this direction.

To provide some context to our work, we give a summary of some of the
known bounds.  The largest minimal free resolution for a monomial
ideal $I$ is the {\it Taylor resolution}~\cite{T}, which is built on the
chain complex of a simplex. Based on the generators of the ideal, one
can build a subcomplex of the Taylor complex called the {\it Scarf
  complex}~\cite{BPS}, whose chain complex is included in the
minimal free resolution of $I$. In other words,
\begin{equation}\label{e:guide} \mbox{Scarf (chain) complex of }I \subseteq
  \mbox{ Minimal free resolution of } I \subseteq
  \mbox{ Taylor (chain) complex of }I. 
\end{equation}
Note that these bounds are sharp: there are monomial ideals whose
minimal free resolutions are Scarf or Taylor. The Taylor resolution,
whose size only depends on the number of generators $q$, therefore
provides effective binomial upper bounds for betti numbers of monomial
ideals. If $I$ is generated by $q$ monomials, this bound is
$$ \beta_i(I)\le \binom{q}{i+1}\qquad\text{for all \, $i$ \, with \,\,
  $0\le i\le q-1$}.
$$

But once powers are taken, the Taylor resolution is never minimal
when $q>1$, and as the powers grow, the binomial bound becomes
unreasonably large. This is where extremal ideals come in: for
integers $r> 1$ \eqref{e:guide} can be refined (and abbreviated) as
\begin{equation}\label{e:guide-2} \Scarf(I^r) \subseteq
  \mbox{ \begin{tabular}{c}Minimal free \\ resolution of  $I^r$
         \end{tabular}}\subseteq
  \mbox{ \begin{tabular}{c}Minimal free \\ resolution of  $\Erq$
         \end{tabular}}
  \subseteq
  \Taylor(\Erq)
\end{equation}
where the last inclusion is strict when $r,q >1$, and
  the middle inclusion should be understood in a sense made more
  precise in \cref{t:upperbound}, where we also show that if a simplicial
  chain complex gives a resolution of $\Erq$, then it also gives a
  resolution of $I^r$.

The statement in \eqref{e:guide-2} leads to the following
questions. {\it Is there a combinatorial description of the minimal
  free resolution of $\Erq$? Can it be described as the chain complex
  of a simplicial complex?}

Most importantly, computational evidence points to the following
questions.
\begin{center}{\it Do the ideals $\Erq$ have Scarf resolutions (i.e. the
    smallest possible)? \\ \medskip Does the Scarf complex of $\Erq$ have a nice
  combinatorial description?}
\end{center}

This first of these two questions is the premise of the current paper, and we provide
evidence towards an affirmative answer. Our focus is finding faces of
the Scarf complex of $\Erq$, which we denote by $\Srq$. If $\Erq$ is indeed
Scarf, the complex $\Srq$ can be considered as the natural $r$-th
power of the Taylor complex on $q$ generators.

The ideal $\E_q$ itself (the case $r=1$) is essentially a {\it nearly Scarf
  ideal} (see~\cite{PV}) corresponding to a $q$-simplex, and hence has
a Scarf resolution. In fact, for $\E_q$ the Scarf and Taylor
resolutions coincide.

In this paper we prove that the ideal ${\E_q}^2$ is also Scarf for any
$q$, as are ${\E_q}^r$ for any $q \leq 4$ and $r \geq 1$ (\cref{c:r=2-Scarf}, \cref{t:Morse-small-q}).  We give
precise combinatorial formulas for the betti numbers of $\Erq$ in
these cases (\cref{upper bounds}, \cref{t:L2r}).  Most of
the input for the work in the case $r=2$ comes from \cite{L2, Lr},
while the case $q\le 4$ brings new techniques.

We also prove that the first betti numbers of $\Erq$ are always Scarf
for any $r,q \geq 1$: in \cref{t:first-betti} we
prove that $\beta_1(\Erq)$ is equal to the number of $1$-dimensional
faces of $\Srq$. When $r=3$, we provide in \cref{t:betti-1} a combinatorial count of such
faces, and thus a formula for $\beta_1({\E_q}^3)$ for any $q$.

Most of this paper is written with a view towards understanding the
Scarf complex $\Srq$. We give a characterization of the faces of
$\Srq$ in terms of polyhedral geometry, and we identify a large number
of facets of $\Srq$ via a subcomplex $\UUrq$ which we introduce in
\cref{d:Ur}. The cases mentioned above where $\Erq$ is a Scarf ideal
coincide with situations when $\Srq=\UUrq$.  While in general we know
that $\Srq$ has more faces than those in $\UUrq$, the simplicial
complex $\UUrq$ allows us to characterize a large number of
multigraded betti numbers of $\Erq$ and provide general lower bounds
for the total betti numbers of $\Erq$ (\cref{c:scarf-betti} and
\cref{t:f-vector}).

  As a step towards showing that $\Srq$ supports a resolution
  of $\Erq$ - in the sense that its simplicial chain complex can be
  homogenized into a free resolution of $\Erq$ - in \cref{t:german} we
  use discrete Morse theory to show that for each facet $U$ of $\Srq$
  which comes from $\UUrq$, one can construct a CW complex supporting a
  resolution of $\Erq$ by eliminating all faces containing $U$.

In view of \eqref{betti}, the computations of the betti numbers
described above also provide effective bounds for the betti numbers of
$I^r$, where $I$ is any ideal generated by $q$ square-free
monomials. While these bounds are too lengthy to be displayed here, we
mention two bounds on projective dimension that come from \cref{upper
  bounds}: If $I$ is any monomial ideal generated by $q$ square-free
monomials, then for any $r\ge 1$,
 $$\pd(I^r)\le 3\qwhen q=3 \qand \pd(I^r) \le 5 \qwhen q=4.$$

Extremal ideals were introduced during our collaborative work~\cite{Lr} with Susan
Cooper, Susan Morey, and Sarah Mayes-Tang.  The BIRS2023 \say{Powers
  of a Simplex}~\cite{juniper} working group have since started
working on developing a proper definition of $\Srq$ using the tools of
discrete geometry and discrete topology.

\section{{\bf \large  Basic Definitions}}
\label{s:basic-definitions}

In this section, basic definitions and background are provided.
Throughout the section, let $\sfk$ be a field, let $S$ be the
polynomial ring $\sfk[x_1, \dots, x_n]$, and let $I = (m_1, \dots,
m_q)$ be an ideal minimally generated by $q$ square-free monomials in
$S$.

\subsection{Simplicial complexes and resolutions}
Given a vertex set $V$, a {\bf simplicial complex} $\Delta$ on a
vertex set $V$ is a set of subsets of $V$ satisfying the following
property: if $\sigma \in \Delta$ and $\tau \subseteq \sigma$, then
$\tau \in \Delta$. An element $\sigma$ of $\Delta$ is called a {\bf
  face}, and a maximal face of $\Delta$ (under inclusion) is called a
{\bf facet}. Since a simplicial complex $\Delta$ can be uniquely
determined by its facets $\sigma_1,\ldots,\sigma_q$, one writes
$\Delta=\langle \sigma_1,\ldots,\sigma_q\rangle$. The {\bf dimension}
of a face $\sigma$ is defined as $\dim(\sigma)=|\sigma|-1$, where the
vertical bars denote cardinality.  If a simplicial complex has only
one facet, then it is called a {\bf simplex}.  In particular, the {\bf
  Taylor complex} $\Taylor(I)$ of $I$ is a simplex on $q$ vertices,
each of which is labeled by a monomial generator of $I$, and each face
is labeled with the lcm of the monomial labels of its generators.

More generally, let $\Delta$ be a simplicial complex on $q$ vertices
$v_1,\ldots,v_q$, and label each vertex $v_i$ with the monomial $m_i$, corresponding to the minimal generators of the ideal $I$.
Label each face of $\Delta$ with the least common multiple of the labels of its
vertices. More precisely, if $\sigma\in \Delta$, then the monomial label of $\sigma$ is 
$$
\m_\sigma=\lcm(m_i\st i\in \sigma).
$$
  
An exact
sequence $F_\bullet$ of free $S$-modules is a {\bf graded free resolution of $I$} if it has the form:
\begin{equation} \label{graded free}
 0 \to F_d \stackrel{\partial_d}{\longrightarrow} \cdots \to F_i
\stackrel{\partial_i}{\longrightarrow} F_{i-1} \to\cdots \to F_1
\stackrel{\partial_1}{\longrightarrow} F_0
\end{equation}
 where $I \cong
F_0/\im(\partial_1)$, and each map $\partial_i$ is graded, in the
sense that it preserves the degrees of homogeneous elements.
The free resolution in \eqref{graded free} is called {\bf minimal} if $\partial_i(F_i) \subseteq (x_1,\ldots,x_n) F_{i-1}$ for every
$i>0$.  To further refine the grading on $F_i$, each free module can be written as direct sum of one dimensional free $S$-modules of the form $S(m)$, indexed by the monomials $m$. Denote by $\LCM(I)$ the  poset consisting of the lcm's of the generating set of $I$ ordered by divisibility; i.e., the lcm-lattice of $I$. Thus, when \eqref{graded free} is minimal, one has
$$F_i \cong \bigoplus_{m \in \LCM(I)} S(m)^{\beta_{i,m}}$$ where the
$\beta_{i,m}$ are the {\bf multigraded betti numbers} of $I$, which
are invariants of $I$.  The {\bf projective dimension}, denoted by
$\pd_R(I)$, is the length of a minimal free resolution of $I$, which
is $d$ in the case of $F_\bullet$ in \eqref{graded free}, with $F_d\ne 0$. 

Given a simplicial complex $\Delta$ with $q$ vertices labelled with
the monomial generators of $I$, the simplicial chain complex of
$\Delta$ can be ``homogenized'' using the monomial labels on the faces
to give a graded complex of free $S$-modules. For details on the
homogenization construction, see \cite{PV}. Let $F^\Delta_\bullet$
denote the homogenized complex obtained from $\Delta$. If
$F^\Delta_\bullet$ is a free resolution of $I$, $\Delta$ is said to
{\bf support a free resolution} of $I$ and the resulting free
resolution is called a {\bf simplicial resolution} of $I$.

The {\bf Scarf complex} of $I$, denoted by $\Scarf(I)$, is a
subcomplex of $\Taylor(I)$ obtained by removing all faces of
$\Taylor(I)$ that share a monomial label with another face. The
relevance of Scarf complexes comes from the fact that a minimal free
resolution of $I$ contains an isomorphic copy of
$F^{\Scarf(I)}_\bullet$ and, in particular, for any face $\sigma\in
\Scarf(I)$ the minimal free resolution of $I$ has a unique generator
with multidegree equal to that of the monomial label of $\sigma$.  An
ideal $I$ is said to be a {\bf Scarf ideal} if $\Scarf(I)$ supports a
free resolution of $I$ (which is necessarily minimal).

The next lemma will be useful in our arguments involving Scarf complexes. 

\begin{lemma}\label{l:expansion} Let $I$ be a monomial ideal with
  Taylor complex $\TT$ and Scarf complex $\SS$, and let $\sigma \in
  \TT$. Then   $\sigma\in \SS$ if and only if both of the following statements hold: 
\begin{enumerate}
\item $\m_\sigma\ne \m_{\sigma\ssm \{v\}}$ for all vertices $v\in \sigma$; 
\item  $\m_{\sigma\cup\{v\}}\ne \m_\sigma$ for all vertices $v\in \TT\ssm\sigma$. 
\end{enumerate}
 In particular, if $\sigma $ is an edge (that is, $\dim(\sigma)=1$), then $\sigma\in \SS$ if and only if  (2) holds. 
\end{lemma}

 \begin{proof} If $\sigma\in \SS$, then (1) and (2) clearly hold. 
Assume now that $\sigma\notin \SS$. We show that (1) or (2) fails to hold.
Since $\sigma\notin \SS$, there exists a face $\tau \in \TT$, such that $\tau \neq \sigma$
    and $\m_\sigma=\m_\tau$.

    If $\tau \subsetneq \sigma$, then there
    exists $v \in \sigma \ssm \tau$.  So $$\tau \subseteq \sigma \ssm
    \{v\} \subset \sigma \qand \m_\tau=\m_\sigma \quad \Longrightarrow \quad 
    \m_\tau=\m_{\sigma\ssm\{v\}}=\m_\sigma.$$

    If $\tau \not \subseteq \sigma$, then there exists  $v \in \tau$
    with $v \notin \sigma$. The fact that $\m_\sigma=\m_\tau$ now
    guarantees that $$\m_{\sigma\cup \{v\}}=\m_\sigma.$$

Thus (1) or (2) do not hold.  Lastly, when $\dim(\sigma)=1$,
$\sigma=\{m_1,m_2\}$ where $m_1$ and $m_2$ are minimal monomial
generators of $I$, and so $\m_\sigma \neq m_1$ and $\m_\sigma \neq
m_2$. Thus (1) holds trivially, and one only needs to check (2) to
establish $\sigma\in \SS$.
    \end{proof}

The next criterion can be used to prove that the Scarf complex supports a minimal free resolution. 

\begin{lemma}{\cite[Lemma 3.1]{BPS}}
\label{l:BPS}
If $I$ is a monomial ideal for which the nonzero betti numbers are concentrated in the multidegrees of the monomial labels of $\Scarf(I)$, then $\Scarf(I)$ supports a minimal free resolution of $I$. 
\end{lemma}

A main technique in this paper relies on an algebraic version of discrete Morse
theory, as developed by E.~Batzies and V.~Welker\cite{BW}, described below. 

\subsection{CW complexes and Morse matchings} \label{s:Morse} 
In general terms, a {\bf CW complex} is a topological space that can
be built inductively by a process of attaching $n$-disks along their
boundary. The interiors of the $n$-disks are the $n$-{\bf cells} of
the CW complex, which can be viewed as a disjoint union of cells. We
refer the reader to \cite{OW} for a precise definition.  As described
there, one can define a notion of free resolution supported on a CW
complex, generalizing the concepts discussed above.  We present below
the basic notions and results that will allow us to use discrete Morse
theory to construct free resolutions.

Let $V$ denote a finite set. We let $2^V$ denote the set of subsets of $V$. 
Let $Y$ be a subset of $2^V$. Then  $G_Y$ denotes  the directed graph whose
vertex set is $Y$, and whose set
of directed edges is
$$  E(G_Y)=\{\sigma \to \sigma\ssm\{v\} \st  v\in \sigma, \sigma\in Y, \sigma\ssm\{v\}\in
  Y\}.
$$
Suppose $\A$ is a matching on $G_Y$, that is, a set of pairwise disjoint
edges of $G_Y$, and let $G_Y^\A$ be a directed graph with the same
vertex set as $G_Y$, and with edges in $\A$ reversed, that is
$$E(G_Y^\A)= \big( E(G_Y)\ssm \A \big ) \cup
\{\sigma \ssm \{v\} \to \sigma  \st \sigma \to \sigma \ssm \{v\} \in \A  \}.$$

The matching $\A$ is {\bf acyclic} if the graph $G_Y^\A$ contains no
directed cycle.

A vertex of $G_Y$ that is not in $\A$ is called an {\bf
  $\A$-critical vertex} of $G_Y$, or an $\A$-{\bf critical cell}  of $Y$.

\begin{lemma} {\cite[Lemma 3.2(1)]{Morse}} \label{inclusions} 
Let $Y,Y'$ be subsets of $2^V$ such that $Y\subseteq Y'$. If
$\A\subseteq E_Y$, then $\A$ is an acyclic matching of $G_{Y'}$ if and
only if $\A$ is an acyclic matching of $G_Y$.
\end{lemma}

\begin{lemma}[{\bf Cluster Lemma} {\cite[Lemma 4.2]{Jo}}]\label{clusterlemma} 
Let $V\subseteq 2^V$, $Q$ a poset and $\{Y_q\}_{q\in Q}$ a partition
of $Y$ such that
$$
\text{If $\sigma\in Y_q$ and $\sigma'\in Y_{q'}$ satisfy $\sigma' \subseteq \sigma$, then $q'\le q$.}
$$
Let $\A_q$ be an acyclic matching on $G_{Y_q}$ for each $q$. Then $\A=\bigcup_{q\in Q}\A_q$ is an acyclic matching on $G_{Y}$. 
\end{lemma}

\begin{lemma}{\cite[Lemma 3.3]{Morse}} \label{matching-lemma}
Let $Y\subseteq 2^V$ and let $v\in V$.  Then
\begin{equation*}
\begin{aligned}
A_Y^{v} = \big\{\sigma\to \sigma'\in E_Y\st v\in \sigma   {\text{ and }} \sigma'=\sigma\ssm \{v\}\big\}
\end{aligned}
\end{equation*}
is an acyclic matching on $G_Y$. 
\end{lemma}

Assume now that $Y$ is a subset of $\Taylor(I)$. If $\A$ is a matching
on $G_Y$, we say that $\A$ is {\bf homogeneous} if
$\m_\sigma=\m_{\sigma'}$ for all $\sigma\to \sigma'\in \A$. 

Batzies and Welker used homogeneous acyclic matchings to build
cellular resolutions of monomial ideals.

\begin{theorem}[\cite{BW}]\label{t:BW} Let $I$ be a monomial ideal and
  $\A$ be a homogeneous acyclic matching on $G_{\Taylor(I)}$. Then there is a
  CW complex  $\X_\A$ supporting a free resolution of $I$, where for
  each $i\geq 0$, the $i$-cells of $\X_\A$ are in one-to-one
  correspondence with the $i$-dimensional $\A$-critical cells of  $\Taylor(I)$. 
\end{theorem}

The one-to-one correspondence in \cref{t:BW} preserves monomial
labels. More precisely, if $\sigma$ is an $\A$-critical face of $\Taylor(I)$, and $\sigma_\A$ is the unique cell of $\X_\A$ corresponding to $\sigma$, then
$$\m_\sigma=\m_{\sigma_\A}.$$ Moreover by (\cite[Proposition
  7.3]{BW}), for $\A$-critical $i$ and $(i-1)$-faces $\sigma$ and
$\sigma'$ of $\Taylor(I)$, the cell $\sigma'_\A$ is a subcell of
$\sigma_\A$ (or in other words $\sigma'_\A \leq \sigma_\A$ ) if and
only if there is a directed path as in \eqref{e:gradient} in $G^\A_{\Taylor(I)}$ starting from
$\sigma=\sigma_0$ and ending at $\sigma'=\sigma_t\ssm\{v_t\}$.

\begin{equation}\label{e:gradient}
   \begin{array}{ccccccccccccc}
  \sigma_0 &&&& \sigma_1 &\ldots&\sigma_{t-1}&&&& \sigma_{t} &&  \\
      &\searrow&&\nearrow &&&&\searrow &&\nearrow &&\searrow & \\ 
  &&\sigma_0\ssm \{v_0\}&&&&&& \sigma_{t-1} \ssm\{v_{t-1}\}&&&&
       \sigma_{t} \ssm \{v_t\} \\ 
  \end{array}
\end{equation}

From the construction of this path it follows that all faces $\sigma_i$
and $\sigma_i\ssm \{v_i\}$ in \eqref{e:gradient}, except for the start face
$\sigma$ and the end face $\sigma'$, must be in the matching $\A$
(\cite[Discussion~6.3]{Morse}).

\begin{remark}
\label{r:BW}
With assumptions as in \cref{t:BW}, if the set of $\A$-critical cells of $\Taylor(I)$ is equal to $\Scarf(I)$, it follows from \cref{l:BPS} that $\Scarf(I)$ supports a minimal free resolution of $I$.  
\end{remark}

\section{{\bf \large  The extremal ideals $\E_q$}}
\label{s:extremal-ideals}

In this section, we recall the definition introduced in \cite{Lr} of a
class of ideals $\E_q$, minimally generated by $q$ square-free monomials, for
which the powers $\Erq$ have maximal betti numbers among the ideals
$I^r$ where $I$ is minimally generated by $q$ square-free
monomials. We then recall and expand some of the relevant results in \cite{Lr}. 

\begin{definition}[{\bf Extremal ideals} Definition~7.1~\cite{Lr}]\label{d:extremal}
Let $q$ be a positive integer. For every set $A$ with $\emptyset \neq
A \subseteq [q]$, introduce a variable $x_A$, and then set $S_{\E}$ to
be the polynomial ring over a field $\sfk$ shown below:
$$S_{\E}=\sfk\big [ x_A \st \emptyset \neq A \subseteq [q] \big ].$$
 For each $i \in [q]$, define a square-free
monomial $\e_i$ in $S_{\E}$ as
  $$\e_i= \prod_{\substack{A \subseteq [q]\\ i \in A}} x_A.$$ The square-free monomial ideal $\E_q =
  (\e_1,\ldots, \e_q)$ is called a {\bf $\pmb{q}$-extremal ideal}.
  \end{definition}

As demonstrated in the example below, when it is unlikely to lead to
confusion, we will simplify the notation by writing $x_1$ for
$x_{\{1\}}$, $x_{12}$ for $x_{\{1,2\}}$, etc., and refer to a
$q$-extremal ideal simply as an extremal ideal.

\begin{example} \label{mainex} When $q=4$, the ideal $\E_4$ is generated
  by the monomials
  $$\begin{array}{ll}
  \e_1&=x_1 x_{12}  x_{13} x_{14} x_{123} x_{124} x_{134} x_{1234},  \\
  \e_2&=x_{2} x_{12} x_{23} x_{24} x_{123} x_{124} x_{234} x_{1234}, \\
  \e_3&=x_{3} x_{13} x_{23} x_{34} x_{123} x_{134} x_{234} x_{1234}, \\
  \e_4&=x_{4} x_{14} x_{24} x_{34} x_{124} x_{134} x_{234} x_{1234} \\
  \end{array}$$
in the ring $\sfk[x_1, x_2, x_3, x_4, x_{12}, x_{13}, x_{14}, x_{23}, x_{24}, x_{34}, x_{123}, x_{124}, x_{134}, x_{234}, x_{1234}]$. 
\end{example}

In general, the ring $S_{\E}$ has $2^{q} - 1$ variables, corresponding
to the power set of $[q]$ (excepting $\emptyset$), and each $\e_i$ is
the product of $2^{q-1}$ variables; in particular, those corresponding
to the subsets of $[q]$ that contain $i$.

Using the terminology of \cite{PV}, one can easily verify that $\E_q$
is precisely the {\it nearly Scarf ideal} of a $q$-simplex, by
corresponding the variable $x_{[q]\ssm B}$ to a face $B$ of the
simplex, and therefore $\E_q$ has the Taylor resolution as its minimal
free resolution. The Taylor resolution is an upper bound for the
minimal free resolution of {\it any} monomial ideal. It turns out, as stated in \cref{t:upperbound} below, that
all powers of $\E_q$ also have a similar property. The rest of this section is dedicated to explaining that simplicial
 resolutions of powers of extremal ideals bound minimal free
 resolutions of powers of all square-free monomial ideals, and to discussing first steps (as initiated in \cite{Lr}) towards finding such resolutions. 

We recall below some of the results of \cite{Lr}, starting with  necessary notation. 

\begin{notation} \label{n:rqs}
For a field $\sfk$, let $S = \sfk[x_1, \dots, x_n]$ and $I$ be an
ideal of $S$ minimally generated by square-free monomials
$m_1,\ldots,m_q$. Let $$r, q \in \NN \qand \ba = (a_1, \dots, a_q),
\bb = (b_1, \dots, b_q) \in \NN^q.$$
Define 
\begin{itemize}
\setlength\itemsep{0.8em}
\item $[q] =\{1, 2, \dots, q\}$
\item $|\ba|=a_1 + \cdots + a_q$
\item $\pmea={\e_1}^{a_1}\cdots {\e_q}^{a_q}$ and $\bm^{\ba}={m_1}^{a_1}\cdots {m_q}^{a_q}$
\item $\Supp(\ba) = \{j \, | \,a_j \neq 0 \}$
\item $\Nrq = \{\bc \in \NN^q \st |{\bc}| =r\}$
\item $\egcd(\pmea, \pmeb) = \e_1^{\min(a_1, b_1)} \cdots
    \e_q^{\min(a_q, b_q)}$
\item $\Trq=\Taylor(\Erq)$
\item $\Srq=\Scarf(\Erq)$. 
\end{itemize}
\end{notation}

The following statement, which follows directly from
\cite[Equation~(7.3.1)]{Lr}, allows us to use $\ba \in \Nrq$ and
$\pmea$ interchangeably to identify a vertex of $\Trq$.

\begin{proposition}[{\bf The vertices of the Taylor complex $\Trq$}]\label{p:vertices}
  If $\ba$, $\bb\in \Nrq$, then $\pmea=\pmeb$ if and only if
  $\ba=\bb$. In particular, the vertices of $\Trq$ are labeled with
  distinct monomials $\pmea$ with $\ba\in \Nrq$.
\end{proposition}

The comparison of a free resolution of  $I^r$,
when $I$ is generated by square-free monomials $m_1,\ldots,m_q$ in
$S$, to a free resolution of $\Erq$ in $\SE$ is done using the ring
homomorphism $\psi_I$ presented below. 

  \begin{definition}[{\bf The ring homomorphism $\psi_I$, \cite[Definition 7.5]{Lr}}]\label{d:psi}
  Let $I$ be an ideal of the
  polynomial ring $S=\sfk[x_1,\ldots,x_n]$ minimally generated by
  square-free monomials $m_1,\ldots,m_q$. For each $k \in [n]$ set
  $$A_k= \{j \in [q] \st x_k \mid m_j\}$$
 and define $\psi_I$ to be the
  ring homomorphism $$\psi_I \colon \SE \to S \qwhere
  \psi_I(x_A)=\begin{cases}
  \displaystyle \prod_{\substack{k\in[n] \\ A=A_k}} x_k & \mbox{if } A=A_k \mbox{ for some
  } k \in [n],\\ 1 & \mbox{otherwise}.
  \end{cases}
$$
\end{definition}

 \begin{lemma}[{\cite[Lemma 7.7]{Lr}}] 
 \label{l:psi}
 Let $I$ be an ideal of the
   polynomial ring $S=\sfk[x_1,\ldots,x_n]$, minimally generated by
   square-free monomials $m_1,\ldots,m_q$. Then
   \begin{enumerate}
   \item[(i)] $\psi_I(\pmea)=\bma$ for each $\ba\in \Nrq$;
   \item[(ii)] $\psi_I(\Erq)S  =I^r$ for every $r\geq 1$;
   \item[(iii)] $\psi_I$ preserves least
     common multiples, that is: 
$$
   \psi_I(\lcm(\pme^{\ba_1}, \dots, \pme^{\ba_t}))=\lcm(\bm^{\ba_1}, \dots, \bm^{\ba_t}) \qforall \ba_1,\ldots,\ba_t \in \Nrq, \quad t\ge 1.
$$     
   \end{enumerate}
 \end{lemma}

An immediate consequence of \cref{l:psi} is the following statement,
that shows the Scarf complex of $\Erq$ contains the Scarf complex of
any $I^r$ where $I$ is minimally generated by $q$ square-free
monomials, using the $\psi$ function. For this statement to make sense
we consider both Scarf complexes as labeled by vectors $\ba \in \Nrq$,
rather than the corresponding monomials $\bma$ or $\pme^{\ba}$.
 
  \begin{proposition}[{\bf $\Scarf(I^r) \subseteq \Scarf(\Erq)$}]\label{p:psi-Scarf}
    Let $r, q \geq 1$ and let $I$ be an ideal minimally generated by
    $q$ square-free monomials $m_1,\ldots,m_q$. Then $\Scarf(I^r) \subseteq \Srq$.
 \end{proposition}

 \begin{proof} Let $\sigma=\{\ba_1,\ldots, \ba_t\} \in \Scarf(I^r)$. If
   for some $\bb_1,\ldots, \bb_s \in \Nrq$,
   $$\lcm(\pme^{\bb_1},\ldots, \pme^{\bb_s})=\lcm(\pme^{\ba_1},\ldots,
   \pme^{\ba_t})$$ then by
   \cref{l:psi} $$\lcm(\bm^{\bb_1},\ldots,
   \bm^{\bb_s})=\lcm(\bm^{\ba_1},\ldots, \bm^{\ba_t})$$ which
   implies that $s=t$ and 
   $$\{\bb_1,\ldots,\bb_s\}=\{\ba_1, \ldots,\ba_t\}.$$ Hence
   $\{\ba_1,\ldots, \ba_t\} \in \Srq=\Scarf(\Erq)$.
 \end{proof}

Now consider a simplicial complex $\Delta$ supporting a free
resolution of $\Erq$. Since  $\Delta$ is a subcomplex of $\Trq$, the
vertices of $\Delta$ are identified by the elements $\ba \in
\Nrq$, where each vertex $\ba$ has monomial label $\pmea$
(\cref{p:vertices}). 

We let $\psi_I(\Delta)$ be the simplicial complex $\Delta$, with each
vertex $\ba$ relabeled by the monomial $\bma$.  While $\psi_I(\Delta)$
may have multiple vertices sharing the same label, we will see below
that its homogenized chain complex supports a free resolution of
$I^r$. In particular, (the simplicial chain complex of) $\Delta$ is an
upper bound for the minimal free resolution of $I^r$ for any $I$
generated by $q$ square-free monomials.

To make this point in a precise manner, we need to introduce a
(non-standard) grading on $\SE$, which we call the {\bf
  $\psi_I$-grading}. Under the standard grading the ring $S$ is
$\ZZ^n$-graded, and the ring $\SE$ is $\ZZ^{2^q-1}$-graded. The the
$\psi_I$-grading adjusts the standard grading on $\SE$ to make
$\psi_I$ into a homogeneous map. In other words, assuming the standard
grading on $S$ (where the degree of each variable is $1$), the
multidegree of the variable $x_A$ is the same as the multidegree of
$\psi_I(x_A)$, and thus the multidegree of $\pmea$ becomes the same as
the the multidegree of $\bm^{\ba}$ for $\ba\in \Nrq$.  To distinguish
between the standard and nonstandard gradings, we let
$\widetilde{\SE}$ denote the ring $\SE$ with the $\psi_I$-grading.
  
  Using the one-to-one correspondence between $\ZZ^n$ and the set of
  monomials $\M(S)$ in the polynomial ring $S$, we can use the
  elements of $\M(S)$ to index the $\ZZ^n$-grading on $S$, and
  similarly for $\M(\SE)$. Therefore, as graded objects, the
decompositions of $\SE$ and $\widetilde{\SE}$ are
$$ \SE          =\bigoplus_{\m\in \M(\SE)} \sfk \m \qand
\widetilde{\SE} =\bigoplus_{\m'\in \M(S)} \Big(
\bigoplus_{\tiny \begin{array}{c}\m\in \M(\SE)\\ \psi_I(\m)=\m' \end{array}}
\sfk \m \Big ).
$$

If $W=\oplus_{\m\in \M(\SE)}W_\m $ is a
$\ZZ^{2^q-1}$-graded $\SE$-module, we denote by $\widetilde{W}$ the $\ZZ^n$-graded
$\widetilde{\SE}$-module
$$\widetilde{W}=
\bigoplus_{\m'\in \M(S)}\Big(
\bigoplus_{\tiny \begin{array}{c} \m\in \M(\SE)\\ \psi_I(\m)=\m'  \end{array}}
W_\m\Big ).
$$
 
 A homomorphism
$\varphi\colon V\to W$ of graded $\SE$-modules then naturally induces a
homomorphism $\widetilde \varphi\colon \widetilde V\to \widetilde W$
of graded $\widetilde{\SE}$-modules.

To summarize, we have described an exact functor from the category of
$\ZZ^{2^q-1}$-graded $\SE$-modules to the category of $\ZZ^n$-graded
$\widetilde{\SE}$-modules.   In particular, if $F_\bullet$
denotes a graded free resolution of $W$ over $\SE$, then $\widetilde
F_\bullet$ denotes the corresponding graded free resolution of
$\widetilde W$ over $\widetilde{\SE}$.

The $\psi_I$-grading allows us to extend the proof of \cite[Theorem
  7.9]{Lr} to \cref{t:upperbound} and give bounds on multigraded
  betti numbers of powers of square-free monomial ideals.

\begin{theorem}[{\bf Extremal ideals bound betti numbers}]\label{t:upperbound}
  Let $q$ and $r$ be positive integers, and let $I$ be an ideal of the
  polynomial ring $S=\sfk[x_1,\ldots,x_n]$ minimally generated by
  $q$ square-free monomials. 
  \begin{enumerate}
\item If $F_\bullet$ is a multigraded free resolution of $\Erq$ over
  $\SE$, then $\widetilde F_\bullet\otimes_{\widetilde{\SE}}S$
  is a multigraded free resolution
  of $I^r$ over $S$, where, in the tensor product, $S$ is regarded as
  a graded $\widetilde{\SE}$-module via the homogeneous homomorphism
  $\psi_I$.

 \item If a simplicial complex $\Delta$ supports a free resolution of
   $\Erq$, then $\psi_I(\Delta)$ supports  a free resolution of $I^r$.
   
\item  If $\m \in \LCM(I^r)$, then $$\displaystyle \beta_{i, \m}(I^r) \le
  \sum_{\tiny \begin{array}{ll}\pme\in \LCM(\Erq)\\
  \psi_I(\pme)=\m
  \end{array}}
  \beta_{i,\pme}(\Erq).$$

  \item (\cite[Theorem 7.9]{Lr}) If $i\ge 0$ then
      $$\beta_i(I^r)\le \beta_i(\Erq).$$
  \end{enumerate}
\end{theorem}
 
\begin{proof}
As in  \cite[Lemma~7.8]{Lr}, with the added improvement of working in the appropriate $\ZZ^n$-graded setting given by the $\psi_I$-grading,  we have
  \begin{equation}\label{e:Tors}
   \widetilde{S_{\E}}/\widetilde{\Erq}\otimes_{\widetilde S_{\E}}S\cong S/I^r\quad\text{and}\quad
  \Tor_i^{\widetilde S_{\E}}(\widetilde{S_{\E}}/{\widetilde\Erq}, S)=0 \qforall i>0.
  \end{equation}
Then (1) follows from here, by computing the $\Tor$ modules using $\widetilde F_\bullet$. 

(2)  Let $F^\Delta_\bullet$,
respectively $F^{\psi_I(\Delta)}_\bullet$, denote the homogenization of
the chain complex of $\Delta$, respectively $\psi_I(\Delta)$. Then
\cref{l:psi} shows that $\widetilde{F^\Delta}_\bullet\otimes_{\widetilde{\SE}}S\cong
F^{\psi_I(\Delta)}_\bullet$. Then (1) implies that $ F^{\psi_I(\Delta)}_\bullet$ is a free resolution of $I^r$. 

Parts (3) and (4) follow directly from (1), by using the free resolution $\widetilde F_\bullet\otimes_{\widetilde{\SE}}S$ of $I^r$ to give upper bounds on the betti numbers of $I^r$. Note that the summation on the right-hand side of the inequality in (3) is equal to $\beta_{i,\bm}(\widetilde{\Erq})$. 
\end{proof}

\cref{t:upperbound} and \cref{p:psi-Scarf} provide motivation to focus our
attention on finding simplicial complexes that support resolutions of
$\Erq$. All evidence points to $\Erq$ being a Scarf ideal for all $q$
and $r$, which means that we expect $\Srq$ to support a minimal free
resolution of $\Erq$. For this reason, significant effort will be spent in later sections on describing the faces of $\Srq$. 

In~\cite{Lr}, the authors described a simplicial complex $\Lrq$ which
supports a free resolution of $I^r$ for any ideal $I$ generated by $q$
square-free monomials in a polynomial ring. In particular, $\Lrq$
supports a free resolution of $\Erq$, and therefore we have the
following inclusions of labelled simplicial complexes (see, for
example, \cite{Mermin}).
\begin{equation}\label{e:first-inclusions}
  \Srq \subseteq \Lrq \subseteq \Trq.
\end{equation}

The complex $\Lrq$ is significantly smaller than $\Trq$, and supports
a minimal resolution when $r=2$.

\begin{definition}[{\bf $\LL^2_q$~\cite{L2, Lr}}]
\label{L2} 
$\LL^2_q$ is the subcomplex of $\TT^{2}_{q}$ with facets 
$$
\{\e_i\e_j\st i\ne j, i,j\in [q]\}\qand
\{\e_i\e_j\st j\in [q]\} \qforeach \quad i \in [q]
\qwhen q>2,$$ and facets
$$\{\e_i\e_j\st j\in [q]\} \qforeach \quad i \in [q] \qwhen q\le 2.$$
\end{definition}

\begin{example}[$\LL^2_3$ vs $\TT^2_3$] 
\label{e:L2-picture}
The complex $\LL^2_3$ corresponding to $q=3$, $r = 2$, and shown in
\cref{f:L32} supports a minimal free resolution of $\E_3^2 =
(\e_1, \e_2, \e_3)^2$, where $\e_i =x_i x_{ij} x_{ik} x_{ijk}$, for
$i,j,k$ with $\{i, j, k\}=\{1,2,3\}$. In contrast, the 5-simplex
$\TT^2_3$ is a five-dimensional polytope with six vertices, fifteen
edges, and twenty triangles.

\begin{figure}
\begin{tikzpicture}
\tikzstyle{point}=[inner sep=0pt]
\node (a)[point,label=above:$\e_1\e_2$] at (0,1) {};
\node (b)[point,label=left:$\e_1\e_3$] at (-1,0) {};
\node (c)[point,label=right:$\e_2\e_3$] at (1,0) {};
\node (d)[point,label=left:$\e_1^2$] at (-1.5,1.5) {};
\node (e)[point,label=right:$\e_2^2$] at (1.5,1.5) {};
\node (f)[point,label=right:$\e_3^2$] at (0,-1) {};
\draw [fill=gray!20](a.center) -- (b.center) -- (c.center);
\draw [fill=gray!20](a.center) -- (b.center) -- (d.center);
\draw [fill=gray!20](a.center) -- (c.center) -- (e.center);
\draw [fill=gray!20](b.center) -- (c.center) -- (f.center);
\draw (a.center) -- (b.center);
\draw (a.center) -- (c.center);
\draw (a.center) -- (d.center);
\draw (a.center) -- (e.center);
\draw (b.center) -- (c.center);
\draw (b.center) -- (d.center);
\draw (b.center) -- (f.center);
\draw (c.center) -- (e.center);
\draw (c.center) -- (f.center);
\draw[black, fill=black] (a) circle(0.05);
\draw[black, fill=black] (b) circle(0.05);
\draw[black, fill=black] (c) circle(0.05);
\draw[black, fill=black] (d) circle(0.05);
\draw[black, fill=black] (e) circle(0.05);
\draw[black, fill=black] (f) circle(0.05);
\end{tikzpicture}\caption{$\LL^2_3$}\label{f:L32} 
\end{figure}

\end{example}

\begin{proposition}[\cite{Lr}]
\label{L2-supports}
$\TT^{1}_{q}=\Taylor(\E_q)$ supports a minimal free resolution of $\E_q$ and $\LL^2_q$ supports a minimal free resolution of ${\E_q}^2$. 
\end{proposition}

However, as $r$ and $q$ grow, the free resolution supported on $\Lrq$,
while still much smaller than the Taylor resolution, veers farther
away from a minimal resolution.

\section{{\bf \large  Faces of the Scarf complex $\Srq$}} 
\label{s:faces}

The focus of this paper is understanding the Scarf complex $\Srq$ of
$\Erq$.  The first step, as taken in this section, is to understand
the monomial labels of the Taylor complex, and describe the faces of
the Scarf complex. We will be using notation from \cref{n:rqs}.

\begin{lemma}[{\bf Divisibility conditions I}]
\label{l:when-divide} 
 Let  $d, q, r \geq 1$ and $\ba_k=(a_{k1}, \dots, a_{kq}),\bb=(b_1, \dots, b_q)\in \Nrq$, for $0\le k\le d$. Then 
 \begin{equation}
 \label{lcm}
 \lcm(\pme^{\ba_1}, \pme^{\ba_2}, \dots, \pme^{\ba_d})=\prod_{\emptyset\ne A\subseteq [q]} {(x_A)}^{\max \left(\sum_{i\in A}a_{1i}, \sum_{i\in A}a_{2i}, \dots, \sum_{i\in A}a_{di}\right)}
 \end{equation}
and 
\begin{equation}
\label{equiv}
\pmeb\mid \lcm(\pme^{\ba_1}, \pme^{\ba_2}, \dots, \pme^{\ba_d})\iff \sum_{i\in A} b_i\le \max \left(\sum_{i\in A}a_{1i}, \dots, \sum_{i\in A}a_{di}\right)\quad \text{ for all\, $\emptyset \ne A\subseteq [q]$.}
\end{equation}
In particular, if $\pmeb\mid \lcm(\pme^{\ba_1}, \pme^{\ba_2}, \dots, \pme^{\ba_d})$,  then 
\begin{equation}
\label{ineq}
\min(a_{1j}, a_{2j}, \dots, a_{dj})\le b_j\le \max(a_{1j}, a_{2j}, \dots, a_{dj})\qquad\text{for all $j\in [q]$.} 
 \end{equation}
\end{lemma}

    \begin{proof}
    Formula \eqref{lcm} follows directly from the definition of the monomials
    $\e_i$, and \eqref{equiv} is an immediate consequence of
    \eqref{lcm}.
 
    We now prove \eqref{ineq}, assuming that $\pmeb\mid
    \lcm(\pme^{\ba_1}, \pme^{\ba_2}, \dots, \pme^{\ba_d})$.  Let $j\in
        [q]$. Note that the inequality on the right in \eqref{ineq} follows by
        using \eqref{equiv} with $A=\{j\}$. The inequality on the
        left follows by using \eqref{equiv} with
        $A=[q]\ssm \{j\}$, and noting that
      $$b_j=r-\sum_{i\in
          [q]\ssm\{j\}}b_i\qquad\text{and}\qquad
        a_{kj}=r-\sum_{i\in [q]\ssm\{j\}}a_{ki}\,. \qedhere
      $$
     \end{proof}

As a result, we can  characterize all the faces of  $\Srq$.

\begin{proposition}[{\bf The faces of $\Srq$}]
\label{p:faces}
Let $d, q, r \geq 1 $ and $\sigma=\{\pme^{\ba_1}, \pme^{\ba_2}, \dots,
\pme^{\ba_d}\}$ be an $d$-dimensional face of $\Trq$. Let $P_{\sigma}$
denote the polyhedron consisting of all $(b_1, \dots, b_q)\in \RR^q$
such that
$$
\begin{cases}
\sum_{i=1}^q b_i=r\\
\sum_{i\in A} b_i\le \max(\sum_{i\in A}a_{1i}, \sum_{i\in A}a_{2i}, \dots, \sum_{i\in A}a_{di}) \qquad \text{for all $ A\subseteq [q]$.}
\end{cases}
$$ 
Then $\sigma\in \Srq$ if and only if $P_{\sigma'}\cap \NN^q=\{\ba_j\st \pme^{\ba_j}\in \sigma'\}$ for all $\sigma'\subseteq \sigma$. 
\end{proposition}

\begin{proof}
Use \cref{equiv} to write
$$
 P_{\sigma'}\cap \NN^q=\{\bb\in \Nrq\st \pme^{\bb}\mid \m_{\sigma'}\}\,.
$$
Thus the condition that  $P_{\sigma'}\cap \NN^q=\{\ba_j\st \pme^{\ba_j}\in \sigma'\}$ for all $\sigma'\subseteq \sigma$ is equivalent to: 
\begin{equation}
\label{e:rewrite}
\text{If $\sigma'\subseteq \sigma$ and  $\pme^{\bb}\mid \m_{\sigma'}$ for some $\bb\in \Nrq$, then $\pme^\bb\in \sigma'$. }
\end{equation}
Since $\pme^\bb\mid \m_{\sigma'}$ if and only if $\m_{\sigma'}=\m_{\sigma'\cup \{\pme^\bb\}}$, we then use \cref{l:expansion} to see that \eqref{e:rewrite} is equivalent to $\sigma\in \Srq$. 
\end{proof}

Next we show that $\Srq$ inherits the faces of $\SS^{r'}_q$ for $r'<r$.
For $r'\in [r]$ and $\ba \in \N_q^{r'}$ and $\sigma\in \TT^{r-r'}_{q}$
we define
\begin{equation}\label{e:e-sigma}
\pmea\sigma=\{\pmea\cdot v\st v\in \sigma\}\in  \TT^r_q. 
\end{equation}
If $Y$ is a subcomplex of $\Trq$, then $\pmea Y$ denotes the subcomplex of $\Trq$ with faces $\pmea\sigma$, where $\sigma\in Y$. 

\begin{proposition}[{\bf Scarf faces}]\label{p:Scarf-face}
  Let $q\ge 1$, $1 \le r'< r$, and $\ba \in \N_q^{r'}$.
  Then $$\sigma\in \SS^{r-r'}_{q}  \iff \pmea\sigma\in \Srq.$$  
 In particular $\pmea \SS^{r-r'}_{q} \subseteq \Srq.$
  \end{proposition}

\begin{proof} We will prove the statement in the case when $\pmea=\e_i$ for some $i$. The general statement will then follow by induction. 

Without loss of generality, assume $i=1$. Let $\sigma = \{\pme^{\ba_1}, \dots, \pme^{\ba_t}\}$, where each $\ba_i
\in \N_q^{r-1}$. Then
$$\e_1\sigma = \{\e_1 \pme^{\ba_1}, \dots,
\e_1\pme^{\ba_t}\} \qand \m_{\e_1 \sigma} = \e_1 \m_{\sigma}.$$
Assume that $\sigma\in  \SS^{r-1}_{q}$.  Let $\bb_1,\ldots,\bb_h \in \Nrq$ and $$\rho = \{\pme^{\bb_1},
\dots, \pme^{\bb_h}\} \in \Trq \quad\text{such that}\quad  \m_{\rho} =
\m_{\e_1\sigma}=\e_1\m_{\sigma}.$$ 
To argue $\e_1\sigma\in \Srq$, one must show $\rho=\e_1\sigma$. We claim that 
$$\e_1 \mid
\pme^{\bb_j} \qforeach\,  j \in [h].$$
Suppose this is not the case, and for some $j \in [h]$,  $\e_1 \nmid
\pme^{\bb_j}$. In other words, for some positive integers
$b_1,\ldots,b_k$, $$\pme^{\bb_j}={\e_{d_1}}^{b_1} \cdots
    {\e_{d_k}}^{b_k} \qwhere 1 < d_1 < \ldots <d_k \leq q.$$ Since
    $x_{\{d_1,\ldots,d_k\}}\mid \e_{d_z}$ for each $z\in [k]$, and
    since $b_1 + \cdots + b_k=r$, we must then
    have $${x_{\{d_1,\ldots,d_k\}}}^{r} \mid \pme^{\bb_j} \mid
    \m_{\rho}=\e_1\m_\sigma.$$ On the other hand, since $1 \notin
    \{d_1,\ldots,d_k\}$,  we have $x_{\{d_1,\ldots,d_k\}}\nmid \e_1$ and hence 
$${x_{\{d_1,\ldots,d_k\}}}^r \nmid \e_1\m_\sigma,$$
a contradiction. Therefore, for each $z\in [h]$ one can write $\pme^{\bb_z}=\e_1\pme^{\bb'_z}$ where
$\bb'_z \in \N_q^{r-1}$, and $$\rho=\e_1 \rho' \qwhere \rho' =
\{\pme^{\bb'_1}, \cdots, \pme^{\bb'_h}\} \in \TT^{r-1}_{q}.$$
We have $$\m_{\e_1 \sigma}=\m_{\rho}=\m_{\e_1 \rho'} \Longrightarrow
  \e_1\m_{\sigma}=\e_1\m_{\rho'} \Longrightarrow
  \m_{\sigma}=\m_{\rho'}.$$ As $\sigma \in \SS^{r-1}_{q}$, we
  conclude that $\sigma=\rho'$ and hence $\e_1\sigma=\e_1\rho'=\rho$,
  which concludes our argument.

Assume now $\e_1\sigma\in\Srq$. Assume $\m_{\sigma}=\m_{\rho}$ for some $\rho\in \TT^{r-1}_{q}$. It follows that
$$
\m_{\e_1\sigma}=\e_1\m_{\sigma}=\e_1\m_{\rho}=\m_{\e_1\rho}
$$
and hence $\e_1\sigma=\e_1\rho$, implying $\sigma=\rho$. Therefore, $\sigma$ is in $\SS^{r-1}_{q}$. 
    \end{proof}

\section{{\bf \large  Facets of the Scarf complex $\Srq$}} 
\label{f:facets}
This section aims to identify a large number of facets of
$\Srq$ by defining a new simplicial complex $\UUrq$ contained in $\Srq$  (\cref{t:f-vector}), and whose facets are also facets of $\Srq$. 
Using the notation in \eqref{e:e-sigma}, we first identify specific faces.

 \begin{definition}[{\bf $\Urq$ and $\Ua$}] \label{d:Urq}
For $r,q\geq 1$, let $\Urq \in \Trq$ denote the set of all square free
   $r$-tuples of the $\e_i$.  In other words
\begin{align*}\Urq& =\{
\e_{i_1}\e_{i_2}\cdots \e_{i_r} : 1 \leq i_1 < i_2 <\ldots<i_r\leq q
\}\\
&= \{ \pmea : \ba \in \N^r_q,\ a_i \in \{0, 1\} {\text{ for all }}i\in [q]
\}.
\end{align*}

For $\ba=(a_1, \dots, a_q)\in \NN^q$ with $|\ba| < r$, set
\begin{align}
\label{e:Ua}
\Ua& =\pmea \U_q^{r-|\ba|} \\ \nonumber
& =\{\pmec\st \bc =(c_1, \dots, c_q)\in \Nrq, a_i\le c_i\le a_{i}+1\quad\text{for all $i\in [q]$}\}.
\end{align}  
\end{definition}

\begin{example} 
Suppose $q = 3$ and $r = 6$, and suppose $$\ba = (1, 2, 0), \quad \bb = (2,
2, 0) \qand \bc = (1, 0, 1).$$Then $|\ba| = 3$, $|\bb| = 4$, and
$|\bc| = 2$. Note that $\U_3^4=\emptyset$ because $4>3$, and hence  $$\Uc =\pmec \U_q^{r-|\bc|} =\e_1\e_3
\U_3^4=\emptyset\,.$$

Next, note that
$$\Ua =\pmea \U_q^{r-|\ba|}= \e_1\e_2^2 \U_3^3 = \e_1\e_2^2\{\e_1\e_2\e_3\} =
\{\e_1^2\e_2^3\e_3\}, $$ and 
$$\Ub =\pmeb \U_q^{r-|\bb|} = \e_1^2\e_2^2 \U_3^2 = \e_1^2\e_2^2\{\e_1\e_2, \e_1\e_3,
\e_2\e_3\} = \{\e_1^3\e_2^3, \e_1^3\e_2^2\e_3, \e_1^2\e_2^3\e_3\}.$$
In particular, $\Ua \subseteq \Ub$.
\end{example}

The next result explains, in particular,  which of the sets $\Ua$ are maximal with respect to inclusion. 

 \begin{lemma}[{\bf The faces $\Ua$}] \label{l:facets}
   Let $r$ and $q$ be positive integers
   integers.  Then
   \begin{enumerate}
   \item $|\Urq| = \binom{q}{r}$.
   \item If $r>q$ then $\Urq = \emptyset$.
   \item If $q\ge 2$, then the sets $\Ua$ with $\ba \in \NN^q$ and $r-q<|\ba| < r$ are  the maximal elements with respect to inclusion among all sets
     $\Ua$ with $\ba\in \NN^q$ and $|\ba|<r$.
   \end{enumerate}
\end{lemma}

 \begin{proof}
  Statements~(1) and~(2) are straightforward.  Assume now $q\ge 2$. For~(3) we show that if
  $$\ba=(a_1, \dots, a_q),\  \bb=(b_1, \dots, b_q) \in \NN^q \qwith
  |\ba|, \ |\bb|<r,$$ then $\Ua\subseteq \Ub$ if and only one of
  the following conditions holds:
\begin{itemize}
\item $|\ba|<r-q$;
\item $|\ba|=r-q$ and $a_i\le b_i\le a_i+1$ for all $i\in [q]$;
\item $|\ba|>r-q$  and $\ba=\bb$. 
\end{itemize}
Indeed, if $|\ba|<r-q$, then $\Ua=\emptyset$. If $|\ba|=r-q$, then
$\Ua=\{\pmea\e_1\e_2\dots \e_q\}$, and $$\Ua\subseteq \Ub \iff \pmeb \mid
\pmea\e_1\e_2\dots \e_q \text{ and } r-q\le |\bb|\iff a_i\le b_i\le a_i+1 \qforall i\in [q].$$ 
In particular, $\Ua\subseteq \U_{\ba'}$, where $\ba'=(a_1+1, a_2, \dots, a_q)$, with  $|\ba'|=r-q+1<r$ (since we assumed $q>1$) and thus $\Ua$ is not maximal. 

Assume now
$|\ba|>r-q$ and $\Ua\subseteq \Ub$. We need to show $\ba=\bb$. Let $i\in [q]$.  There exists $\pmec\in \Ua$ with such that $c_i=a_{i}+1$.  Since $\pmec\in \Ub$, \eqref{e:Ua} implies $c_i\le
b_i+1$ and hence $a_i\le b_i$. There also exists $\pmed\in
\Ua$ with $\bd=(d_1,
\dots, d_q)$ such that $d_i=a_i$.  Since $\pmed\in \Ub$, we must have $b_i\le d_i$ and hence
$b_i \le a_i$. We conclude $a_i=b_i$ for all $i\in [q]$, so $\ba=\bb$,
as desired.
\end{proof}

We want to define now a simplicial complex $\UUrq$ whose faces are all the subsets of the sets $\Ua$ defined above. We describe this complex by identifying its facets. \cref{l:facets} shows that our definition below makes sense.

\begin{definition}[{\bf The simplicial complex $\UUrq$}]
\label{d:Ur}
For integers $r\geq 1$ and $q\ge 2$ let $\UUrq$ denote the simplicial
complex described by its facets, as follows:
$$
\UUrq=\left\langle \Ua \st \ba \in \NN^q, \ r-q<|\ba| < r \right\rangle \,.
$$
 If $q=1$ , we set $\UUrq=\langle \U_{(r-1)}\rangle$, where $ \U_{(r-1)}=\{\e_1^r\}={\e_1}^{r-1}\U_1^1$.  
\end{definition}

\begin{example}
We write an explicit description of the facets of $\UUrq$ in a few cases.

If  $r=2$ and $q\ge 3$, then  
\begin{equation}\label{e:U2}
\UU^2_q=\left \langle \U^{2}_q,  \e_i\,\U^{1}_q\st i\in [q]\right\rangle=\left\langle \{\e_j\e_k\st 1\le j<k\le q\}, \{\e_i\e_l\st l\in [q]\} \st i\in [q]\right\rangle. 
\end{equation} 
When $q\le 2$, then $\UU^2_q=\left \langle \e_i\,\U^{1}_q\st i\in [q]\right\rangle$. 

If $r=3$ and $q\ge 4$, then 
\begin{align*}
\UU^3_q&=\left \langle \U^{3}_q,  \e_i\,\U^{2}_q, \e_j\e_k\,\U^1_q\st i,j,k\in [q]\right\rangle\\
&=\left\langle \{\e_a\e_b\e_c\st 1\le a<b<c\le q\}, \{\e_i\e_u\e_v\st u,v\in [q], u<v\}, \{\e_j\e_k\e_w\st w\in [3]\} \st i,j,k\in [q]\right\rangle. 
\end{align*}
Also, $\UU^3_3=\left \langle \e_i\,\U^{2}_q, \e_j\e_k\,\U^1_q\st i,j,k\in [q]\right\rangle$ and $\UU^3_q=\left \langle \e_j\e_k\,\U^1_q\st j,k\in [q]\right\rangle$ when  $q\le 2$.

If $r=4$ and $q=3$, then 
\begin{align*}
\UU^4_3&=\left \langle \e_i\e_j\,\U^{2}_3,  \e_a\e_b\e_c\,\U^{1}_3\st i,j ,a,b,c\in [3]\right\rangle\\
&=\left\langle \{\e_i\e_j\e_u\e_v\st 1\le u<v\le 3\}, \{\e_a\e_b\e_c\e_k\st k\in [3] \} \st i,j,a,b,c\in [q]\right\rangle. 
\end{align*}
A geometric realization of the simplicial complex $\UU^4_3$ in 
\cref{f:U-pic} shows that $\UU^4_3$ has $16$ triangle facets.

\begin{figure}  
  \tiny{
 \begin{tikzpicture}
\coordinate (1111) at (0, 0);
\coordinate (1112) at (1, 0);
\coordinate (1122) at (2,0);
\coordinate (1222) at (3, 0);
\coordinate (2222) at (4, 0);

\coordinate (1113) at (0.5, 0.865);
\coordinate (1123) at (1.5, 0.865);
\coordinate (1223) at (2.5,0.865);
\coordinate (2223) at (3.5, 0.865);

\coordinate (1133) at (1, 1.73);
\coordinate (1233) at (2, 1.73);
\coordinate (2233) at (3, 1.73);

\coordinate (1333) at (1.5, 2.595);
\coordinate (2333) at (2.5, 2.595);

\coordinate (3333) at (2, 3.46);

\draw [fill=gray!20](1111) -- (1112) -- (1113);
\draw [fill=gray!20](1123) -- (1112) -- (1113);
\draw [fill=gray!20](1123) -- (1133) -- (1113);
\draw [fill=gray!20](1123) -- (1133) -- (1233);
\draw [fill=gray!20](1333) -- (1133) -- (1233);
\draw [fill=gray!20](1333) -- (2333) -- (1233);
\draw [fill=gray!20](1333) -- (2333) -- (3333);

\draw [fill=gray!20](1112) -- (1122) -- (1123);
\draw [fill=gray!20](1223) -- (1122) -- (1123);
\draw [fill=gray!20](1123) -- (1223) -- (1233);
\draw [fill=gray!20](2233) -- (1223) -- (1233);
\draw [fill=gray!20](2233) -- (2333) -- (1233);

\draw [fill=gray!20](1122) -- (1222) -- (1223);
\draw [fill=gray!20](2223) -- (1222) -- (1223);
\draw [fill=gray!20](2223) -- (1223) -- (2233);
\draw [fill=gray!20](2223) -- (1222) -- (2222);

\draw[black, fill=black] (1111) circle(0.05);
\draw[black, fill=black] (1122) circle(0.05);
\draw[black, fill=black] (1112) circle(0.05);
\draw[black, fill=black] (1222) circle(0.05);
\draw[black, fill=black] (2222) circle(0.05);
\draw[-] (1111) -- (2222);

\node[label = below :$\e_1^4$] at (1111) {};
\node[label = below :$\e_1^3\e_2$] at (1112) {};
\node[label = below :$\e_1^2\e_2^2$] at (1122) {};
\node[label = below :$\e_1\e_2^3$] at (1222) {};   
\node[label = below :$\e_2^4$] at (2222) {};      

\draw[black, fill=black] (1113) circle(0.05);
\draw[black, fill=black] (1123) circle(0.05);
\draw[black, fill=black] (1223) circle(0.05);
\draw[black, fill=black] (2223) circle(0.05);
\draw[-] (1113) -- (2223);
\node[label = left :$\e_1^3\e_3$] at (1113) {};
\node[label = below :$b$] at (1123) {};
\node[label = below :$c$] at (1223) {};   
\node[label = right :$\e_2^3\e_3$] at (2223) {};   

\draw[black, fill=black] (1133) circle(0.05);
\draw[black, fill=black] (1233) circle(0.05);
\draw[black, fill=black] (2233) circle(0.05);
\draw[-] (1133) -- (2233);
\node[label = left :$\e_1^2\e_3^2$] at (1133) {};
\node[label = below :$a$] at (1233) {};   
\node[label = right :$\e_2^2\e_3^2$] at (2233) {};  

\coordinate (1333) at (1.5, 2.595);
\coordinate (2333) at (2.5, 2.595);
\draw[black, fill=black] (1333) circle(0.05);
\draw[black, fill=black] (2333) circle(0.05);
\draw[-] (1333) -- (2333);
\node[label = left :$\e_1\e_3^3$] at (1333) {};
\node[label = right :$\e_2\e_3^3$] at (2333) {};

\draw[black, fill=black] (3333) circle(0.05);
\node[label = above :$\e_3^4$] at (3333) {};

\draw[-] (1111) -- (3333);
\draw[-] (1112) -- (2333);
\draw[-] (1122) -- (2233);
\draw[-] (1222) -- (2223);
\draw[-] (1112) -- (1113);
\draw[-] (1122) -- (1133);
\draw[-] (1222) -- (1333);
\draw[-] (2222) -- (3333);
 \end{tikzpicture} 
  }
\caption{$\UU_3^4$ with $a=\e_1\e_2\e_3^2$, $b=\e_1^2\e_2\e_3$ and $c=\e_1\e_2^2\e_3$}\label{f:U-pic}  
\end{figure}

\end{example}

\begin{remark}[{\bf $\UU_q^2=\LL_q^2$}]
\label{U=L}
Comparing \cref{L2} and \eqref{e:U2}, we see $\UU_q^2=\LL_q^2$. Then
\cref{L2-supports} implies that $\UU_q^2$ supports a minimal free
resolution of ${\E_q}^2$.
\end{remark}

We now start building the steps towards \cref{t:f-vector}, which
shows that the facets of $\UUrq$ are also facets of $\Srq$, and as a
result we establish effective lower bounds on the betti numbers of
$\Erq$.

\begin{lemma}[{\bf Monomial labels on $\UUrq$}]\label{l:necessary}
  Assume $r\leq q$ so that $\Urq \neq \emptyset$. Then the following hold.
\begin{enumerate}

\item The monomial label of the square-free face $\Urq$ is
  \begin{equation}\label{e:murq}
  \m_{\Urq}=\prod_{\emptyset\ne A \subseteq [q]} {x_A}^{\min(|A|,r)}=\frac{\e_1\e_2 \ldots \e_q}{\displaystyle \prod_{\tiny\begin{array}{c}A \subseteq [q]\\ r <|A| 
    \end{array}} {x_A}^{|A|-r}}.
  \end{equation}
\item If $v=\e_{i_1}\e_{i_2}\cdots \e_{i_r} \in \Urq$ for $ 1 \leq
    i_1 < i_2 <\ldots<i_r\leq q$, then
  $$
  \m_{\Urq\sm\{v\}}=\displaystyle 
   \frac{\m_{\Urq}}{x_{\{i_1,\ldots, i_r\}}}\,.
 $$
 \item If $\ba\in \mathbb N^q$ with $r-q<|\ba|<r$, then
$$
   \m_{\Ua}
   =\pmea\cdot \prod_{\tiny \emptyset\ne A \subseteq [q]} {x_A}^{\min(|A|,r-|\ba|)}
   =\frac{\pmea \e_1\e_2 \ldots \e_q}{\displaystyle
     \prod_{\tiny\begin{array}{c}A \subseteq [q]\\ r-|\ba| <|A| \end{array}} {x_A}^{|A|-(r-|\ba|)}}.
$$
\end{enumerate}
\end{lemma}

 \begin{proof} 
Let $\J$ denote the set of all subsets of $[q]$ with $r$
elements. Then $\J$ is in one-to-one correspondence to $\Urq$. If
$J\in \J$ with $J=\{j_1, \dots, j_r\}$, then
$$
\e_{j_1}\dots \e_{j_r}= \prod_{\emptyset\ne A\subseteq [q]}x_A^{|A\cap J|}\,.
$$
We have then 
 \begin{align*}
   \m_{\Urq}
&= \lcm \Big ( \prod_{\emptyset\ne A\subseteq [q]}x_A^{|A\cap J|} \st  J\in \J \Big )
=\prod_{\tiny \emptyset\ne A \subseteq [q]} {x_A}^{\min(|A|,r)} 
=\prod_{\tiny \begin{array}{c} \emptyset\ne A \subseteq [q]\\ |A| \leq r
   \end{array}}  {x_A}^{|A|}
   \cdot
   \prod_{\tiny \begin{array}{c} A \subseteq [q]\\ |A| > r
     \end{array}} {x_A}^{r}\\
&=\frac{\displaystyle \prod_{\tiny \emptyset\ne A \subseteq [q]} {x_A}^{|A|}}
   {\displaystyle \prod_{\tiny \begin{array}{c} A \subseteq [q]\\ |A| > r \end{array}} {x_A}^{|A|-r}}
 =\frac{\e_1\e_2 \ldots \e_q}{\displaystyle \prod_{\tiny \begin{array}{c} A \subseteq [q]\\ |A| > r \end{array}} {x_A}^{|A|-r}}.
 \end{align*}

 Next, given $v$ as in the hypothesis, set $B=\{i_1, \dots, i_r\}$. Since
 $q>r$, for every nonempty $A \subseteq [q]$ we have
$$\max\{|A\cap J|\st J\in \J, J\ne B\}=\begin{cases}
\min(|A|,r) &\text{if $A\ne B$}\\
r-1 &\text{if $A=B$}
\end{cases}
$$
and then 
 \begin{align*}
   \m_{\Urq\ssm\{v\}}
&=\lcm \Big ( \prod_{\emptyset\ne A\subseteq [q]}x_A^{|A\cap J|} \st  J\in\J, J\ne B \Big )
 = \prod_{\tiny \emptyset\ne A \subseteq [q]} {x_A}^{\max\{|A\cap J|\st J\in \J, J\ne B\}} \\
&= {x_B}^{r-1}
\prod_{\tiny \begin{array}{c}\emptyset\ne A \subseteq [q]\\A\ne B\end{array}} {x_A}^{\min(|A|,r)}.
 \end{align*}

 The last statement follows now from \eqref{e:murq}.
   \end{proof}

\begin{proposition}\label{p:german}
  Let $q, r\geq 1$. If  $\pmeb \notin \Ua$ for some $\ba
 \in \NN^q$ with $r-q < |\ba| < r$ and $\bb \in \Nrq$, then 
 then there
 exists $\bc \in \Nrq$ such that
  \begin{equation}\label{e:german}  
  \pmec \in \Ua \qand \m_{\Ua \cup \{\pmeb\}}=\m_{(\Ua \cup \{\pmeb\})\sm \{\pmec\}}.
\end{equation} 
\end{proposition}

\begin{proof}  
  Suppose $\ba=(a_1,\dots, a_q)$ and $\bb=(b_1,\dots, b_q)$.  Set $r'=r-|\ba|$.  The hypothesis implies $r'<q$.

Assume first $\egcd(\pmea, \pmeb)=1$, and hence $\Supp(\ba)\cap
\Supp(\bb)=\emptyset$.  After a possible reordering, without loss of
generality assume for some $t$ and $h$ with $1 \leq t < h \le q$
  $$
  \pmeb={\e_1}^{b_1}\cdots {\e_t}^{b_t}
  \qand
  \pmea={\e_{h}}^{a_{h}}\cdots {\e_{q}}^{a_q}, 
$$ where $a_i, b_j >0$ when $j\leq t$ and $h\le i\le q$, and
$$b_1 + \cdots + b_t = r \qand 
a_{h} + \cdots + a_q = r-r'.$$
   
  Set $\pmec=\pmea w$ where $$w=\e_1\e_2\ldots \e_{r'}\qand A=\{1,
  \ldots, r'\}.$$ We need to show that
     $$
     \pmea w\mid \lcm\left(\m_{\Ua \sm \{ \pmea w\}}, \pmeb \right).
     $$
     In other words, for every nonempty subset $B$ of $[q]$,  and $s\ge 0$, we
     need to show the implication
\begin{equation}
\label{implication}
      {x_B}^s\mid  \pmea w \implies {x_B}^s\mid
      \lcm \left(\m_{\Ua \sm \{ \pmea w\}}, \pmeb \right)
\end{equation}
where by \cref{l:necessary} 
   \begin{equation}\label{e:fraction}
      \m_{\Ua \sm \{ \pmea w\}}
      = \frac{\pmea \m_{\U_q^{r'}}}{x_A}.
      \end{equation} 

Assume ${x_B}^s\mid \pmea w$. If $B\ne A$, then ${x_B}^s$ divides the
fraction in \eqref{e:fraction}, and hence \eqref{implication}
holds. Assume now $B=A$. The largest $s$ such that ${x_A}^s\mid \pmea
w$ is
$$
s=r'+a, \qwhere
a=\begin{cases}
a_{h}+\dots +a_{r'} &\text{if $r'\ge h$}\\
0 &\text{if $r'<h$.}
\end{cases}
$$
If $r'\le t$, then $r'<h$, hence $s=r'$. We have  ${x_A}^{r'} \mid \pmeb$, hence \eqref{implication} holds. 

 Assume now $r' >t$. Since
 $b_1+\dots +b_t=r$, we have ${x_A}^r\mid \pmeb$. Since $a\le r-r'$, it
 follows that ${x_A}^{r'+a}\mid \pmeb$, and thus \eqref{implication} is
 established, finishing the proof under the assumption that
 $\egcd(\pmea, \pmeb)=1$.

Next, assume $\egcd(\pmea, \pmeb)=\pmed \ne 1$, where $\bd \in \N_q^d$
for some $d$ with $0 < d \le r-r'$.  Write $\pmea= \pmed \cdot \pme^{\ba'}$ and
$\pmeb= \pmed \cdot \pme^{\bb'}$, with $\ba' \in \N_q^{r-r'-d}$, $\bb' \in \N_q^{r-d}$  and $\egcd(\pme^{\ba'}, \pme^{\bb'})=1$. Using the argument
above, choose $w\in \U_q^{r'}$ such that
 \begin{equation*}
  \m_{\U_{\ba'} \cup \{\pme^{\bb'}\}}=\m_{(\U_{\ba'} \cup \{\pme^{\bb'}\})\sm \{\pme^{\ba'} w\}}.
  \end{equation*}
The equality \eqref{e:german} follows from here. 
      \end{proof}

We are now ready to state the main theorem in this
section. Recall that the $\f$-vector of a simplicial
complex $\Delta$ is the tuple $\f(\Delta)=(f_i)_{i\ge 0}$ with $f_i$
denoting the number of $i$-dimensional faces of $\Delta$.

\begin{theorem}[{\bf Facets$(\UUrq) \subseteq$ Facets$(\Srq)$}]\label{t:f-vector}
The set $\Ua$ is a facet of of the Scarf complex $\Srq$ for all
$\ba\in \NN^q$ such that $r-q<|\ba|<r$. In particular,
$\UUrq\subseteq \Srq$.

Furthermore, for all $i\ge 0$, $$f_i\le
\beta_i^S(\Erq)$$ where $\f(\UUrq)=(f_i)_{i\ge 0}$.
\end{theorem}

\begin{proof} We first show $\Urq$ is a face of $\Srq$.
If $u, v$ are vertices of $\Trq$
   with $u \in \Urq$ and $v \notin \Urq$, then we observe that by \cref{l:necessary}
   $$\m_{\Urq \sm \{u\}}\ne \m_{\Urq}.$$
Moreover we claim $$\m_{\Urq \cup \{v\}}\ne \m_{\Urq}.$$
To see this assume, without loss of generality, that  
$$v ={\e_{1}}^{a_1}\cdots {\e_{t}}^{a_t} \in \Trq \sm \Urq,
\quad a_1>1. $$ Since $a_1>1$, it must be that $t<r$.  It suffices to
show that $v\nmid \m_{\Urq}$. Indeed, ${x_{[t]}}^r\mid v$, but the
largest power of $x_{[t]}$ that divides $\m_{\Urq}$ is
${x_{[t]}}^t$. The conclusion follows thus from the inequality $t<r$.

Using \cref{l:expansion}, it follows that $\Urq \in \Srq$. The fact
that $\Ua\in \Srq$ now follows from \cref{p:Scarf-face}. To see that
it is a facet, apply \cref{p:german} with $r'=r-|\ba|$, noting that
$r'\le q-1$ because $r-|\ba|<q$.
\end{proof}

An immediate consequence is that ${\E_q}^2$ is a Scarf ideal.

\begin{corollary}[{\bf ${\E_q}^2$ is Scarf}]
\label{c:r=2-Scarf} 
If $q\ge 1$, then $\UU^2_q=\SS^2_q$. In particular the ideal
${\E_q}^2$ is Scarf.
\end{corollary}

\begin{proof}
By \cref{U=L}, we know that $\UU^2_q$ supports a minimal free
resolution of $\E_q^2$. Since $\UU^2_q\subseteq \SS^{2}_{q}$, it
follows that $\UU^{2}_{q}=\SS^{2}_{q}$.
\end{proof}

Another consequence of \cref{t:f-vector} is \cref{c:scarf-betti} below, 
which identifies some of the multi-graded betti numbers of $\Erq$.

\begin{corollary}[{\bf Multigraded betti numbers from $\UUrq$}]
  \label{c:scarf-betti}
Let $r\ge 1$, $q\ge 1$ and $\ba\in \NN^q$ with $r-q<|\ba|<r$. Then 
\[
\beta_{i,\m_{\Ua}}(\Erq)=1 \qwhere i=\binom{q}{r-|\ba|}.
\]
\end{corollary}

As mentioned in the introduction, we believe that $\Srq$ always
supports a minimal free resolution of $\Erq$. The next result provides
support in this direction, by showing that for every facet of $\Srq$
of the form $\Ua$, there is a CW complex that supports a free
resolution of $\Erq$ with a cell corresponding to $\Ua$ and no cells
corresponding to faces of $\Trq$ which strictly contain $\Ua$.

\begin{theorem}[{\bf A cellular resolution of $\Erq$ based on  $\Ua$}]
  \label{t:german}
  If $q,r \geq 1$, and $\ba\in \NN^q$ with $r-q<|\ba|<r$, then
  there is a matching $\A$ on the faces of $\Trq$ whose critical cells form   a CW complex $\X$ such that
  \begin{itemize}
  \item $\X$ supports a resolution of $\Erq$;
  \item $\X$ contains (an isomorphic copy of) the simplex $\Ua$;
  \item the cells of $\X$ are in (monomial label-preserving)
    one-to-one correspondence with the set of faces of $\Trq$ that do
    not strictly include $\Ua$.
  \end{itemize}
\end{theorem}

\begin{proof} 
Let  $\ba\in \NN^q$ with $r-q<|\ba|<r$. Set 
$$\Ta=\{\sigma \in \Trq \st \Ua \subsetneq \sigma\}.$$

  \begin{Claim} For $\sigma, \tau \in \Ta$ and $w\in \Ua$ and $w'\in \tau$, if $\sigma\ssm\{w\}=\tau\ssm\{w'\}$, then
    $\sigma=\tau$.
  \end{Claim}

    To prove the Claim, consider the fact that both $\sigma$ and $\tau$
    contain $\Ua$, so $\sigma \cap \Ua=\tau \cap \Ua =\Ua$ and 
  \[
\Ua\ssm\{w'\}=(\tau\ssm\{w'\})\cap \Ua=(\sigma\ssm\{w\})\cap \Ua=\Ua\ssm\{w\}\,.
  \]
 Since $w\in \Ua$, we must have  $w=w'$ and hence $\sigma=\tau$. This proves the claim. 
   \medskip

For each $\sigma \in \TT_\ba$, use \cref{p:german} to pick
  $w_\sigma \in \Ua$ such that
  \begin{equation}\label{e:homogeneous}
  \m_{\sigma}=\m_{\sigma \sm \{w_\sigma\}}.
  \end{equation}

  We will prove the desired statement by constructing a homogeneous acyclic matching on $G_{\Trq}$. We set
$$\A=\{\sigma\to \sigma\ssm\{w_\sigma\}\in G_{\Trq}\colon \sigma\in \TT_\ba\}.
  $$ The Claim shows that $\A$ is a matching, and
  \eqref{e:homogeneous} ensures that the matching $\A$ is homogeneous.

We now show that $\A$ is acyclic. A cycle $\C$ will have edges
alternating between those in $\A$ and those outside $\A$. In this
picture, the up arrows are edges in $\A$ (reversed) and the down
arrows are of the type $\sigma\to \sigma\ssm\{v\}\notin \A$.
   
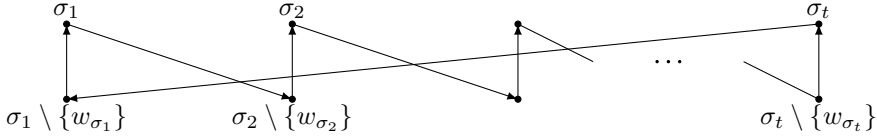
\begin{tikzpicture}[label distance=-5pt] 
\coordinate (A) at (-6, 0);
\coordinate (B) at (-6, 1);\\
\coordinate (C) at (-3, 0);
\coordinate (D) at (-3,  1);
\coordinate (G) at (0,0);
\coordinate (H) at (0, 1);
\coordinate (I) at (1, .5);
\coordinate (J) at (3, .5);
\coordinate (L1) at (1.85, .5);
\coordinate (L2) at (2, .5);
\coordinate (L3) at (2.15, .5);
\coordinate (M) at (4, 0);
\coordinate (N) at (4,  1);
\coordinate (Z) at (.5, -.5);
 \draw[black, fill=black] (A) circle(0.04);
\draw[black, fill=black] (B) circle(0.04);
\draw[black, fill=black] (C) circle(0.04);
\draw[black, fill=black] (D) circle(0.04);
 \draw[black, fill=black] (G) circle(0.04);
 \draw[black, fill=black] (H) circle(0.04);
 \draw[black, fill=black] (M) circle(0.04);
 \draw[black, fill=black] (N) circle(0.04);
 \draw[black, fill=black] (L1) circle(0.015);
 \draw[black, fill=black] (L2) circle(0.015);
 \draw[black, fill=black] (L3) circle(0.015);
\draw[-latex] (A) -- (B);
\draw[-latex] (B) -- (C);
\draw[-latex] (C) -- (D);
\draw[-latex] (D) -- (G);
\draw[-latex] (G) -- (H);
\draw[-] (H) -- (I);
\draw[-] (J) -- (M);
\draw[-latex] (M) -- (N);
\draw[-latex] (N) -- (A);
\node[label = below :$\sigma_1 \sm \{w_{\sigma_1}\}$] at (A) {};
\node[label = above : $\sigma_1$] at (B) {};
\node[label = below :$\sigma_2 \sm \{w_{\sigma_2}\}$] at (C) {};
\node[label = above : $\sigma_2$] at (D) {};
\node[label = below :$\sigma_t \sm \{w_{\sigma_t}\}$] at (M) {};
\node[label = above : $\sigma_t$] at (N) {};
\end{tikzpicture}

We have thus $\sigma_t\ssm\{v\}=\sigma_1\ssm\{w_{\sigma_1}\}$ for some
$v\in \sigma_t$. Then the Claim implies $\sigma_1=\sigma_t$, a
contradiction to $\C$ being a cycle.

As a result, $\A$ is a homogeneous acyclic matching, and therefore
by \cref{t:BW} there is a CW complex $\X_{\A}$ which supports a
resolution of $\Erq$, and its cells are in (label-preserving) one-to-one correspondence with the $\A$-critical cells of $\Trq$.

Now note that since $\Ua$ and all of its faces have Scarf monomial
labels, they cannot be vertices of any edges in  $\A$. Therefore for every $\sigma
\subseteq \Ua$, $\sigma$ is $\A$-critical and every subface of $\sigma$
is $\A$-critical by the same reasoning.

Also since no Scarf face can appear in a path as in \eqref{e:gradient}
the only subcells of $\sigma_{\A}$ are $\tau_{\A}$ where $\tau \subseteq
\sigma$. This proves that $(\Ua)_{\A}$ is (isomorphic to) a simplex
in $\X_{\A}$. 
\end{proof}

\begin{remark}
  It is possible to extend the matching in \cref{t:german} and remove
  faces of the Taylor complex containing multiple $\Ua$'s at the same
  time, or even almost all of them. This is somewhat beyond the scope
  of this paper, so we do not include it. However, it provides more
  evidence that the ideals $\Erq$ may be Scarf ideals.
\end{remark}

\section{{\bf \large  $\Erq$  is a Scarf ideal when  $q\le 4$}} \label{s:small-q}
\label{s:small-q}
This section is dedicated to proving in \cref{t:Morse-small-q} the statement in the title. We
also obtain a  concrete description of the Scarf complex of $\Erq$ when
$q\le 4$, namely $\Srq=\UUrq$. 

We start with  an application of \cref{l:when-divide} in
establishing divisibility relations that are ingredients in our main result.  

\begin{lemma}[{\bf Divisibility conditions II}]  
\label{l:some-divisibility}
 Let $r,q\ge 1$ and $\ba_1, \ba_2\in \Nrq$. 
 The following hold:
\begin{enumerate}
\item If $u\in \Supp(\ba_1)$, $k\in\Supp(\ba_2)$ and $\Supp(\ba_1)=\{u\}$ or $\Supp(\ba_2)=\{k\}$ , then 
$$
\frac{\pme^{\ba_1}\e_k}{\e_u}\mid \lcm(\pme^{\ba_1}, \pme^{\ba_2})\,.
$$
\item Assume $\Supp(\ba_1)=\{u,v\}$ and $\Supp(\ba_2)=\{k,l\}$ with $u\ne v$ and $k\ne l$.  Then $$\frac{\pme^{\ba_1}\e_k\e_l}{\e_u\e_v}\mid \lcm(\pme^{\ba_1}, \pme^{\ba_2})\,.$$
\end{enumerate}
\end{lemma}

\begin{proof}
Write $\ba_1=(a_{11}, a_{12},\dots, a_{1q})$ and $\ba_2=(a_{21}, a_{22}, \dots, a_{2q})$. Recall from \cref{l:when-divide} that, if  $\bb=(b_1, \dots, b_q)\in \Nrq$, then 
\begin{equation}
\label{e:show}
\pmeb\mid \lcm(\pme^{\ba_1}, \pme^{\ba_2})\iff \sum_{i\in A}b_i\le \max(\sum_{i\in A}a_{1i}, \sum_{i\in A}a_{2i})
\end{equation}
for all $A$ with $\emptyset\ne A\subseteq [q]$. 

Let $A$ be as above. We proceed to prove the inequality in \eqref{e:show} under the given hypotheses. 

(1) We may assume $k\ne u$, since the statement is trivial otherwise. 

 Let $\bb=(b_1, \dots, b_q)\in \Nrq$ so that $\pmeb=\displaystyle \frac{\pme^{\ba_1}\e_k}{\e_u}$. 
We have  then
$$
\sum_{i\in A}b_i=\sum_{i\in A}a_{1i}-|A\cap\{u\}|+|A\cap\{k\}|\,.
$$
This expression is less than or equal to $\sum_{i\in A}a_{1i}$ when $u\in A$ or $k\notin A$, and hence it suffices to assume $u\notin A$ and $k\in A$.  We further observe 
 \begin{align*}
\Supp(\ba_1)=\{u\}&\,\implies\,\sum_{i\in A}b_i=1\le \sum_{i\in A}a_{2i}\\
\Supp(\ba_2)=\{k\}&\,\implies\, \sum_{i\in A}b_i\le r=\sum_{i\in A}a_{2i}\,,
\end{align*}
establishing thus the inequality in \eqref{e:show}.

  (2) We may assume that $u,v,k,l$ are all distinct, since otherwise the statement reduces, after factoring out a common factor, to (1). 

 Let $\bb\in \Nrq$ such that
       $\pmeb={\displaystyle\frac{\pme^{\ba_1}\e_k\e_l}{\e_u\e_v}}$. 
    We have then
$$
 \sum_{i\in A}b_i=\sum_{i\in A} a_{1i}-|A\cap
          \{u,v\}|+|A\cap \{k,l\}|\,.
$$
This expression is less than or equal to $\sum_{i\in A}a_{1i}$ when $|A\cap \{k,l\}|\le |A\cap\{u,v\}|$ and hence it suffices to assume $|A\cap \{k,l\}|> |A\cap\{u,v\}|$.  We further have: 
\begin{align*}
|A\cap  \{k,l\}|=2 &\,\implies\,  \sum_{i\in A}b_i\le r= \sum_{i\in A}a_{2i}\\
 |A\cap \{k,l\}|=1 \quad\text{and}\quad |A\cap \{u,v\}|=0 &\,\implies\, \sum_{i\in A}b_i=1\le  \sum_{i\in A}a_{2i}\,,
\end{align*}
establishing thus the inequality in \eqref{e:show}.
 \end{proof}
      
Using the same notation as in \cref{n:rqs}, we introduce a
lexicographic order on $\Nrq$, with
\begin{equation}\label{e:order}
\ba > \bb \iff \quad\text{there exists $i\in[q]$ with $a_j=b_j$ for all $j<i$ and $a_i > b_i$}.
\end{equation}

We induce this order on the generators of $\Erq$ by 
$$\pmea > \pmeb \iff \ba>\bb.$$

\begin{lemma} \label{l:find-c}
  Assume $\ba,\bb\in \Nrq$ are such that
  \begin{equation}\label{e:conditions}
|\Supp(\ba)\cup\Supp(\bb)|\le 4,\qquad    \ba>\bb \qand |a_s-b_s|>1 \qforsome s\in [q].
  \end{equation}
  Then there exists $\bc\in \Nrq$ such that the following hold: 
\begin{enumerate}
\item $\bc\ne \bb$ and $\ba>\bc$;
 \item $\{\pmec, \pmea\}\in \Srq$; 
 \item $\pmec\mid\lcm(\pmea, \pmeb)$.  
\end{enumerate}
In particular, $\{\pmea, \pmeb\}\notin \Srq$.
\end{lemma}

\begin{proof}
  First note that the hypothesis implies $r\geq 2$, otherwise $|a_s-b_s|\leq 1$ for all $s \in [q]$.

Assume first  $\egcd(\pmea,\pmeb)=1$, so that $\Supp(\ba)\cap \Supp(\bb)=\emptyset$. 
  \begin{case1} $\pmea=\e_i^r$ for some $i\geq 1$. The assumption and the hypothesis \eqref{e:conditions}
  imply
   $$b_k=0 \qfor k\le i, \qquad b_j>0 \qforsome
    j>i.$$ Let $\bc\in \Nrq$ so that $$\pmec=\e_i^{r-1}\e_j.$$ Observe $\pmea>\pmec$, since $j>i$. Also,   $\pmec \neq \pmeb$ since $b_i=0$ and $r\ge 2$. We have 
$$\{\pmea, \pmec\}=\e_i^{r-1} \{\e_i,\e_j\}\in\UUrq$$ and hence (2) holds, by  using \cref{t:f-vector}.  Moreover, $\pmec\mid \lcm(\pmea,\pmeb)$  by   \cref{l:some-divisibility} , so (3) 
    follows.
\end{case1}
  
  \begin{case2} $\pmeb=\e_i^r$ for some $i\ge 1$. The assumption and the  hypothesis \eqref{e:conditions}
  imply
    $$i \geq 2, \qquad
      a_j>0 \qforsome j<i, \qquad
      a_i=0.$$
      Let $\bc\in \Nrq$ so that $$\pmec=\frac{\e_i\pmea}{\e_j}\,.$$
  Observe $\pmea>\pmec$, since $j<i$. Also, $\pmec \neq \pmeb$ since $c_i =1\leq r-1 <b_i$.  We have 
$$\{\pmea, \pmec\}=\frac{\pmea}{\e_j} \{\e_i,\e_j\}\in 
    \UUrq $$ 
and hence (2) holds, by  using \cref{t:f-vector}.    Moreover, $\pmec\mid \lcm(\pmea,\pmeb)$  by
    \cref{l:some-divisibility} , so (3) 
    follows.
\end{case2}

  \begin{case3}  $|\Supp(\ba)|>1$ and $|\Supp(\bb)|>1$. Since $|\Supp(\ba)\cup\Supp(\bb)|\le 4$, we must have  $|\Supp(\ba)|=|\Supp(\bb)|=2.$
    By \eqref{e:conditions} we must therefore have
    $$\pmea=\e_i^{a_i}\e_j^{a_j} \qand
     \pmeb= \e_k^{b_k}\e_l^{b_l} \qwhere  a_i,a_j,b_k,b_l>0,$$
$i,j,k,l$ are distinct, $i>k,j,l$ and  $\max(a_i,a_j,b_k,b_l)\ge 2$. 
Take $\bc\in \Nrq$ so that 
 $$
 \pmec=\frac{\pmea\cdot \e_k\e_l}{\e_i\e_j}=\e_i^{a_i-1}\e_j^{a_j-1}\e_k\e_l\,.
 $$
Observe $\ba>\bc$, since $a_i-1<a_i$. Since $\max(a_i,a_j,b_k,b_l)\ge 2$, we also have $\bc\ne \bb$, hence  (1) holds. 

It follows that
  $$\{\pmea, \pmec \}=
  \e_i^{a_i-1}\e_j^{a_j-1}\{\e_i\e_j,\e_k\e_l\}\in \UUrq$$
  and hence (2) holds, by using \cref{t:f-vector}.   Finally,  (3) follows from \cref{l:some-divisibility}.
\end{case3}

This finishes the proof in the case  $\egcd(\pmea,\pmeb)=1$. In general, let 
 $\egcd(\pmea,\pmeb)=\pmed$ so that
    $$\pmea=\pmed\pmeaa \qand \pmeb=\pmed\pmebb \qquad \text{for}\quad  \ba',\bb'\in \N_q^{r-|\bd|}\,.$$
It follows that  $\egcd(\pmeaa,\pmebb)=1$. The hypothesis \eqref{e:conditions} implies $|\Supp(\ba')\cup\Supp(\bb')|\le 4$, $\ba'>\bb'$ and $|a'_s-b'_s|>1$ for some $s\in [q]$. Applying the first part of the proof for $\ba'$ and $\bb'$ implies there exists $\bc'\in \N_q^{r-|\bd|}$ such that
$$\bc'\ne \bb', \qquad \ba'>\bc', \qquad \{\pmecc,\pmeaa\}\in \SS_q^{r-|\bd|},\quad \text {and}\quad  \pmecc\mid \lcm(\pmeaa, \pmebb)\,.
$$
 Taking $\bc\in \Nrq$ so that $\pmec=\pmed\pmecc$ and using \cref{p:Scarf-face} for (2), we see that this choice for $\bc$ satisfies the desired conclusions. 
\end{proof}

\begin{proposition} 
\label{p:U=Scarf}
Assume $1\le q\le 4$, $r\ge 1$. The following are equivalent for $\sigma\in \Trq$: 
\begin{enumerate}
\item $\sigma\in \Srq$;
\item $\{\pmea, \pmeb\}\in \Srq$ for all $\pmea,\pmeb \in \sigma$;
\item $|a_i-b_i|\le 1$ for all $i\in [q]$ and all $\pmea,\pmeb \in \sigma$;
\item $\sigma\subseteq \Uc$, where $\pmec=\egcd(\sigma)$.
\end{enumerate}
\end{proposition}

\begin{proof}
The implication (1)$\implies$(2) is clear. The implication (2)$\implies$(3) follows directly from \cref{l:find-c}. The implication (4)$\implies$(1) follows from \cref{t:f-vector}. 

We show now (3)$\implies$(4). Assume  $\pmea \in \sigma$. To show $\pmea \in \Uc$, we need to show $\pme^{\ba-\bc}\in \U_q^{r-|\bc|}$, or, equivalently, $a_i-c_i\in \{0,1\}$ for all $i\in [q]$. This follows from (3), since $c_i=\min\{a_i\st \pmea \in \sigma\}$. 
\end{proof}

Recall that when $q>1$, $\UUrq$
was defined as the cell complex with facets $\Uc$, for $\pmec\in
\NN^q$ and $r-q<|\bc|<r$. We now have a full characterization of the
Scarf complex when $q\le 4$:

\begin{corollary} 
\label{c:U=S}
If $r \geq 1$ and   $1 \leq q \leq 4$, then $\Srq=\UUrq$.\qed
\end{corollary} 

 \begin{remark}
\label{r:q>4}
 When $q>4$, it is no longer true that $\UU_q^r=\Srq$. We check below that, when $q\ge 5$, then  $$\{\e_1^2\e_2, \e_3\e_4\e_5\}\in \SS^{3}_{q} \ssm \UU^{3}_{q}.$$
Indeed, set $\m=\lcm( \e_1^2\e_2,  \e_3\e_4\e_5)$. We have
 \begin{align*}
\e_1^2\e_2&=\prod_{\tiny \{1,2\}\subseteq A\subseteq [q]} {x_A}^3  \prod_{\tiny \begin{array}{c} A \subseteq [q]\\1\in A, 2\notin A\end{array}} {x_A}^2 \prod_{\tiny \begin{array}{c}A \subseteq [q]\\2\in A, 1\notin A \end{array}}x_A\\
 \e_3\e_4e_5&=\prod_{\tiny \{3,4,5\}\subseteq A\subseteq [q]} {x_A}^3  \prod_{\tiny \begin{array}{c} A \subseteq [q]\\|A\cap \{3,4,5\}|=2\end{array}} {x_A}^2 \prod_{\tiny \begin{array}{c} A \subseteq [q]\\|A\cap \{3,4,5\}|=1\end{array}}x_A\,.
 \end{align*}
Observe
\begin{align}
\label{a1}x_A^3\mid \m & \iff \{1,2\} \subseteq A\qor \{3,4,5\}\subseteq A\\
\label{a2} x_A^2\mid \m &\iff 1\in A\qor |A\cap \{3,4,5\}|\ge 2\\
\label{a3}x_i\mid \m &\iff i\in [5]\,.
\end{align}

Assume that $\e_i\e_j\e_k\mid \m$ with $\e_i\e_j\e_k\notin \{\e_1^2\e_2,
\e_3\e_4\e_5\}$ and $i\le j\le k$. Using \eqref{a3}, we see $i,j,k\in [5]$.

We cannot have $i=j=k$, because ${x_i}^3$ does not divide $\m$, cf.\! \eqref{a1}.  If $i,j,k$ are all distinct, then
${x_{ijk}}^3\mid \m$ and it follows from \eqref{a1}  that $i=1$, $j=2$ and $k\in
\{3,4,5\}$. But then ${x_{2k}}^2\mid \m$, contradicting \eqref{a2}.  If exactly
two of $i,j,k$ are equal, we must have $i=j=1$ and $k\in \{3,4,5\}$,
because $x_s^2\mid \m$ implies $s=1$ by \eqref{a2}. But then ${x_{1k}}^3\mid \m$, again a 
contradiction by \eqref{a1}. 

In view of \cref{l:expansion}, this establishes the desired conclusion. 
 \end{remark}

We now proceed to build up preliminaries towards proving that $\Srq$ supports a (minimal) free resolution of $\Erq$, when $q\le 4$.

 \begin{lemma} \label{l:d} 
Assume $r\geq 1$ and $1\le q\le 4$,  $\ba, \bb, \bc, \bd\in \Nrq$ and
$$\{\pmea, \pmed\}, \{\pmeb, \pmed\}\in \Srq \qand
\pmec\mid \lcm(\pmea, \pmeb). 
$$
Then $\{\pmec, \pmed\}\in \Srq$. 
 \end{lemma}
 
 \begin{proof}
 Assume $\{\pmea, \pmed\}, \{\pmeb, \pmed\}\in
 \Srq$. Then by \cref{p:U=Scarf}
 $$
 |a_i-d_i|\le 1\qand |b_i-d_i|\le 1\qforall i\in[q].
 $$
 Equivalently, 
 $$
 d_i-1\le a_i, b_i\le d_i+1\qforall i\in [q].
 $$
 Since $\pmec\mid \lcm(\pmea, \pmeb)$, \cref{l:when-divide} implies that
 $$
 \min(a_i,b_i)\le c_i\le \max(a_i,b_i) \qforall i\in [q],
 $$
 and hence 
 $$
 d_i-1\le c_i\le d_i+1\qforall i\in [q].
 $$
 Therefore, $|c_i-d_i|\le 1$ for all $i\in [q]$ and thus $\{\pmec, \pmed\}\in \Srq$. 
 \end{proof}

We are now ready to introduce a homogeneous acyclic matching on the
face poset of $\Trq$, when $q\leq 4$, which will result in a minimal
free resolution of $\Erq$ supported on $\Srq$. For this purpose, we
will need the following notation.

\begin{notation}
Let
  \begin{equation}\label{e:S}
    \MS=\{(\ba, \bb)\st \ba, \bb\in \Nrq, \ba>\bb, \{\pmea,
\pmeb\}\notin \Srq\}.
  \end{equation}
On this set, we introduce the following order
 relation, recalling \eqref{e:order}
  \begin{equation}\label{e:pair-order} (\ba, \bb)\ge (\bc, \bd)\iff \ba > \bc \text \qor \ba=\bc \mbox{ and 
    }\bb\ge \bd.
  \end{equation}

\begin{itemize}
\item For $\sigma\notin \Srq$ (note that $|\sigma| \geq 2$) let
  \begin{equation}\label{e:nu}
  \nu(\sigma)=\max \{(\ba, \bb)\in \MS \st \pmea, \pmeb\in
  \sigma\}.
  \end{equation}
  
\item For $(\ba, \bb)\in \MS$ let
    \begin{equation}\label{e:Sab}\MS_{\ba, \bb}=\{\sigma\in \Trq\ssm \Srq\st
  \nu(\sigma)=(\ba, \bb)\}.
    \end{equation}

\item For $\sigma\in \Trq \ssm \Srq$ and $\nu(\sigma)=(\ba, \bb)$, let
  $\iota(\sigma)$ be the largest (under the order in \eqref{e:order})
  $\bc\in \Nrq$ satisfying the conclusion of \cref{l:find-c}, that is,
    \begin{equation}\label{e:iota}\iota(\sigma)=\max \{\bc \in \Nrq \st \ba > \bc \ne \bb, \quad \{\pmea,
  \pmec\}\in \Srq \qand \pmec\mid \lcm(\pmea, \pmeb)\}.
    \end{equation}
\end{itemize}
\end{notation}
\begin{remark}
Note that by \cref{p:U=Scarf} the sets $\MS_{\ba,\bb}$, with
$\ba,\bb\in \Nrq$, $\ba>\bb$ and $\{\pmea, \pmeb\}\notin \Srq$ form a
partition of $\Trq\ssm \Srq$,
\begin{equation}\label{e:partition}
  \Trq\ssm \Srq=
\bigcup_{\tiny \begin{array}{c}\ba,\bb\in \Nrq, \ba>\bb\\ \{\pmea, \pmeb\}\notin \Srq \end{array}}
\MS_{\ba,\bb}.
\end{equation}
\end{remark}

 \begin{proposition}\label{p:inS}
Assume $r\ge 1$, $1 \leq q \leq 4$, and $\sigma \in \Trq\ssm \Srq$.

 If $\bc=\iota(\sigma)$, then  $\sigma\cup\{\pmec\}\notin \Srq$ and $\sigma\ssm\{\pmec\}\notin \Srq$ and 
\begin{enumerate}
\item $\nu(\sigma\cup\{\pmec\})= \nu(\sigma)=\nu(\sigma\ssm\{\pmec\})$;
\item $\iota(\sigma\cup\{\pmec\})= \iota(\sigma)=\iota(\sigma\ssm\{\pmec\})$;
\item $\m_{\sigma\cup\{\pmec\}}=\m_\sigma=\m_{\sigma\ssm\{\pmec\}}$.
\end{enumerate}

\end{proposition}
 
 \begin{proof}
   Suppose
   $$\nu(\sigma)=(\ba, \bb)\in \MS, \quad
   \iota(\sigma)=\bc \qand
   \sigma'=\sigma\cup\{\pmec\}.$$
   Since $\sigma\notin \Srq$, we also have $\sigma'\notin \Srq$. Also, $\bc\ne \ba$, $\bc\ne \bb$ by \eqref{e:iota}, and, since $\{\pmea, \pmeb\}\notin \Srq$, it follows that $\sigma\ssm\{\pmec\}\notin \Srq$. 

(1)  Suppose $\nu(\sigma')=(\ba', \bb')$, so that by \eqref{e:S} and
 \eqref{e:nu} we have $\{\pmeaa, \pmebb\} \notin \Srq$.

 If $\ba'<\ba$, then $\nu(\sigma')<(\ba, \bb)$, a contradiction.

 Assume $\ba'> \ba$. Since $\ba>\bc$, we have $\ba'\ne \bc$.  If $\bb' \ne \bc$, then $\pmeaa, \pmebb\in \sigma$ and thus $\nu(\sigma)=
 (\ba',\bb') \neq (\ba, \bb)$, a contradiction. We must have $\bb'
 = \bc$, and therefore
 $$\nu(\sigma')=(\ba', \bc)
 \qwith \ba'>\ba, \quad \pmeaa\in \sigma.
 $$ 
Using the order \eqref{e:pair-order}, we see $(\ba',\ba)>(\ba,\bb)$ and $(\ba',\bb)>(\ba, \bb)$. Further, since $\pmeaa\in \sigma$ and $\nu(\sigma)=(\ba,\bb)$, the definitions in \eqref{e:Sab} and \eqref{e:S} give
 $$\{\pmeaa, \pmea\}, \{\pmeaa, \pmeb\}\in\Srq.$$
 Then \cref{l:d} implies $\{\pmeaa, \pmec\}\in \Srq$, a contradiction.

 We conclude  $\ba'=\ba$.   If $\pmebb\notin\sigma$, then $\bb'=\bc$ and $\{\pmea, \pmec\}\notin\Srq$, a
 contradiction. So $\pmebb\in \sigma$, implying $\bb'=\bb$, and this concludes the proof of the first equality in (1). 
 
 The second equality in (1) follows directly from \eqref{e:nu}, since $\bc\ne \ba$ and $\bc\ne \bb$. 

(2)   Since $\iota(\sigma')$ only depends on $\nu(\sigma')=\nu(\sigma)$,
 this equality follows immediately.

(3) This statement follows directly from \eqref{e:iota}.
\end{proof}
 
 \begin{theorem}[{\bf A Scarf resolution for $\Erq$ when $q \leq 4$}]
   \label{t:Morse-small-q}
Assume $1 \le q\le 4$ and $r\geq 1$.  Using the notation as in \eqref{e:cab} and
\cref{s:Morse}, let $G=G_{\Trq}$ and let $\A$ denote the following
subset of $E(G)$
 $$
 \A=\big\{\sigma \to \sigma \ssm \{\pme^{\iota(\sigma)}\} \st \sigma\in \Trq\ssm \Srq \qand  \pme^{\iota(\sigma)}\in\sigma\big\}
 $$
 Then $\A$ is a homogeneous acyclic matching of $G$ and its set of critical cells is $\Srq$.

 Consequently, $\Srq$ supports a minimal free resolution on $\Erq$. 
 \end{theorem}
 
  \begin{proof}
  We first show that $\A$ is a matching. By \cref{p:inS}, for every
    $\sigma \in \Trq \ssm \Srq$ with $\pme^{\iota(\sigma)} \in \sigma$ we 
    have
 $$\iota(\sigma)=\iota(\sigma')\qquad\text{where}\quad\sigma'=\sigma \ssm \{\pme^{\iota(\sigma)}\}).$$ 
In particular, 
    $\pme^{\iota(\sigma')} \notin \sigma'$, and hence each face of $\Trq \ssm \Srq$ gets matched
    exactly once, thus $\A$ is a matching.

    The fact that $\A$ is homogeneous follows also from \cref{p:inS}.

    It remains to show that $\A$ is acyclic. For this purpose we will
    use \cref{matching-lemma} applied to the partition of $\Trq\ssm
    \Srq$ in \eqref{e:partition} into sets $\MS_{\ba,\bb}$, with
    $(\ba, \bb)\in \MS$.

    For $(\ba,\bb)\in \MS$ as above, define
 $$
 \A_{\ba,\bb}=\big\{\sigma \to   \sigma \ssm \{\pme^{\iota(\sigma)}\} \st  \sigma\in \MS_{\ba,\bb} \qand \pme^{\iota(\sigma)}\in \sigma \big\}.
 $$

 We proceed with two claims.
 
 \begin{Claim1} $\A_{\ba,\bb}$ is a homogeneous acyclic matching on the
   induced subgraph $G_{\MS_{\ba,\bb}}$ of $G$ on the vertices in
   $\MS_{\ba,\bb}$ and its set of critical cells is empty.
\end{Claim1}

If $\sigma \in \MS_{\ba,\bb}$, then $\bc=\iota(\sigma)$ only depends
on $(\ba, \bb)$. In other words if $\tau \in \MS_{\ba,\bb}$, $\iota(\tau)=\bc$
 as well.
 
With the notation in \cref{matching-lemma}, we see that
$$\A_{\ba, \bb}=\A^{v}_{\MS_{\ba,\bb}} \quad\text{with}\quad v=\pmec.$$ Then
\cref{matching-lemma} shows that $\A_{\ba, \bb}$ is an acyclic
matching. It is also homogeneous, since $\A$ is.
Moreover, observe that  every $\sigma \in \MS_{\ba,\bb}$ is matched as
$$\begin{cases}
  \sigma \to \sigma \ssm\{\pmec\} \in \A_{\ba, \bb} & \qif \pmec \in \sigma, \\
  \sigma \cup \{\pmec \} \to \sigma  \in \A_{\ba, \bb} & \qif \pmec \notin \sigma.
\end{cases}
$$
Indeed, to see that   $\sigma \cup \{\pmec \} \to \sigma  \in \A_{\ba, \bb}$, observe that $\iota(\sigma\cup \{\pmec\})=\bc$ by \cref{p:inS}, and hence $(\sigma\cup \{\pmec\})\ssm \{\pme^{\iota(\sigma\cup \{\pmec\})}\}=\sigma$. 

Therefore, the set of $\A_{\ba, \bb}$-critical cells of this matching is empty.

 \begin{Claim2}
  If $\sigma'\in \MS_{\ba, \bb}$ and $\sigma\in \MS_{\bc, \bd}$ with
  $\sigma' \subseteq \sigma$, then $(\ba, \bb)\le (\bc, \bd)$.
 \end{Claim2}
 
 Indeed, the hypotheses in the Claim~2 imply
 $$
 (\ba,\bb)=\nu(\sigma') \leq \nu(\sigma)=(\bc, \bd).
 $$
 Now we apply \cref{clusterlemma}, observing that 
 $$
 \A=\bigcup_{(\ba, \bb)\in \MS}\A_{\ba, \bb}\qquad\text{and} \qquad\Trq\ssm \Srq=\bigcup_{(\ba, \bb)\in \MS}\MS_{\ba,\bb}\,.
 $$ Claim 1 and Claim 2 show that $\A$ is a homogeneous acyclic
 matching of $G_{X'}$, where $X'=\Trq\ssm \Srq$, and its set of
 critical cells is empty.  Finally, from \cref{inclusions} it 
 follows that $\A$ is a homogeneous acyclic matching of $G$, with the
 set of $\A$-critical cells equal to $\Srq$.
 
The  conclusion that $\Srq$ supports a minimal free resolution of $\Erq$ follows then from \cref{r:BW}.
\end{proof}

A direct consequence of \cref{t:Morse-small-q} and \cref{c:U=S} is that the $\bf f$-vector of $\UUrq$ gives the betti numbers of $\Erq$ when $q\le 4$. 

\begin{corollary}
\label{c:f-betti}
If $1\le q\le 4$ and $r\ge 1$, then $\beta_i(\Erq)=f_i$ for all $i\ge 0$, where ${\bf f}(\UUrq)=(f_i)_{i\ge 0}$. 
\end{corollary}

In \cref{s:betti}, we will record explicit formulas for these betti 
numbers. The example below is a preview.

  \begin{example} 
\label{e:q=3r=4}When $q=3$ and $r=4$, the picture in \cref{f:U-pic} is a geometric realization of the simplicial complex $\Srq=\UUrq$. By \cref{t:Morse-small-q}, we know that this complex supports a minimal free resolution of ${\E_3}^4$. In particular, we see that $\pd({\E_3}^4)=2$ and,  counting the number of vertices, edges and triangles, we get: 
$$
\beta_0({\E_3}^4)=15 \qquad \beta_1({\E_3}^4)=30\qquad \beta_2({\E_3}^4)=16\,.
$$
If desired, the multi-graded betti numbers can also be explicitly described, along with the differentials in the minimal free resolution supported on this complex.  
     \end{example}

\section{{\bf \large  The first betti numbers of $\Erq$ are Scarf for all $r$ and $q$}}
\label{s: first}
We will prove in \cref{t:first-betti} that there is a multigraded free
resolution of $\Erq$ with only Scarf multidegrees appearing in the
first homological degree. Since the Scarf multidegrees appear in
{\it every} free resolution, this will imply
that $$\beta_{1,\m}(\Erq) \neq 0 \iff \m \mbox{ is the monomial label
  of an edge of } \Srq.$$

To find such a complex, we describe  in \cref{t:first-betti} a matching that removes all the
non-Scarf edges in $\Trq$. We first need some
preliminary steps, using \cref{n:rqs}.

\begin{proposition} \label{p:edges}
Let $r,q\ge 1$ and $\ba, \bb, \bc\in \Nrq$.  If $\lcm(\pmea,
\pmeb)=\lcm(\pmea, \pmec)$, then $\pmeb=\pmec$.
 \end{proposition}
 \begin{proof}
   Assume $\lcm(\pmea, \pmeb)=\lcm(\pmea, \pmec)$, and let $i\in [q]$
   be such that $b_i \neq c_i$.  In view of \cref{l:when-divide},
   $\pmec \mid \lcm(\pmea, \pmeb)$ gives
\begin{equation}
\label{e:1}
 \min(a_i,b_i)\le c_i\le \max(a_i,b_i),
\end{equation}
and $\pmeb\mid\lcm(\pmea, \pmec)$ gives
\begin{equation}
\label{e:2}
\min(a_i,c_i)\le b_i\le \max(a_i,c_i).
\end{equation}

Without loss of generality, assume $b_i \gneq c_i$. Then by \eqref{e:2}
$\max(a_i,c_i)=a_i$ and hence $b_i\le a_i$. Therefore
$\min(a_i,b_i)=b_i$ and \eqref{e:1} implies $b_i\le c_i$, a
contradiction. Hence $b_i=c_i$ for all $i \in [q]$, and therefore
$\pmeb=\pmec$.
 \end{proof}

 When  $\{\pmea, \pmeb\} \in \Trq\ssm \Srq$, \cref{l:expansion}
    guarantees the existence of an element $\bc_{\ba,\bb} \in \Nrq$
    such that 
      \begin{equation}\label{e:cab}
        \lcm(\pmea,\pmeb,\pme^{\bc_{\ba,\bb}})=\lcm(\pmea,\pmeb) \qand
        \bc_{\ba,\bb} \notin \{\ba, \bb\}
      \end{equation}
      We next show that there is a CW complex supporting a free
      resolution of $\Erq$, whose edges correspond to the edges of the
      $\Srq$. This is done by defining a matching on $G=G_{\Trq}$ (see
      \cref{s:Morse} for notation).  While we need to make a choice of
      $\bc_{\ba,\bb}$ in order to define a matching, the argument
      below holds for any set of choices that one makes.  For the
      theorem below, we choose a fixed element $\bc_{\ba,\bb}$
      satisfying \eqref{e:cab} for each $\ba,\bb \in \Nrq$.
      
\begin{theorem}{\bf (The 1st betti numbers of $\Erq$ are all Scarf.)}
\label{t:first-betti} 
Assume $r,q \geq 1$, and using notation as in \eqref{e:cab} and
\cref{s:Morse}, let $G=G_{\Trq}$ and let $\A$ denote the following
subset of $E(G)$
$$
\A=\left\{\{\pmea,\pmeb,\pme^{\bc_{\ba,\bb}}\}\to \{\pmea,\pmeb\} \in E(G)\st \{\pmea,\pmeb\}\notin \Srq  \right\}. $$
Then
\begin{enumerate}
\item The set $\A$ is a homogeneous acyclic matching of $G$; 

\item The set of $\A$-critical edges of $\Trq$ is precisely the set of
  edges of $\Srq$;

\item There exists a CW complex $\X_{\A}$ that supports a resolution of
  $\Erq$ whose edges are in one-to-one correspondence with the edges
  of $\Srq$;

\item The first betti number $\beta_1(\Erq)$ is equal to the
  number of edges of $\Srq$;

\item $\m$ is a multidegree with $\beta_{1,\m}(\Erq)\ne 0$ if and only
  if $\m=\lcm(\m',\m'')$, where $\m'$ and $\m''$ are monomial labels
  of an edge of $\Srq$.
\end{enumerate}
\end{theorem}

\begin{proof}
 (1) We first show that $\A$ is a matching. Given the fixed choice of
  $\bc_{\ba,\bb}$, the vertices in $\A$ of the form $\{\pmea,\pmeb\}$
  each belong to only one edge of $\A$. Assume now that a
  $2$-dimensional face $\sigma =\{\pmea, \pmeb, \pmed\}$ of $\Trq$ is a vertex
  of two edges of $\A$, say $$\sigma\to \{\pmea, \pmeb\} \qand
  \sigma\to \{\pmea, \pmed\}$$ which implies, by \eqref{e:cab},
  that $$\lcm(\pmea, \pmeb)=\lcm(\sigma)=\lcm(\pmea, \pmed).$$ Then by
  \cref{p:edges} $\pmeb=\pmed$, a contradiction.

  By \eqref{e:cab} the matching $\A$ is homogeneous.

  Finally, assume that there is a cycle $\C$ in $G$, by
  \cite[Lemma~3.3]{FFGY}, $\C$ (shown below) will have at least 6
  edges alternating between those in $\A$ and those outside $\A$, and
  all having the same monomial label.  
 
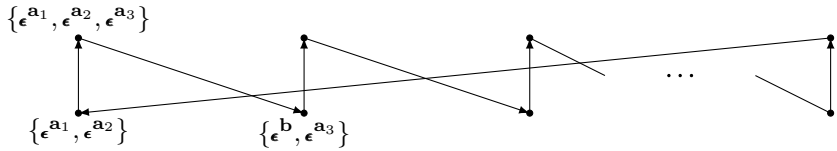
\begin{tikzpicture}[label distance=-5pt] 
\coordinate (A) at (-6, 0);
\coordinate (B) at (-6, 1);\\
\coordinate (C) at (-3, 0);
\coordinate (D) at (-3,  1);
\coordinate (G) at (0,0);
\coordinate (H) at (0, 1);
\coordinate (I) at (1, .5);
\coordinate (J) at (3, .5);
\coordinate (L1) at (1.85, .5);
\coordinate (L2) at (2, .5);
\coordinate (L3) at (2.15, .5);
\coordinate (M) at (4, 0);
\coordinate (N) at (4,  1);
\coordinate (Z) at (.5, -.5);
 \draw[black, fill=black] (A) circle(0.04);
\draw[black, fill=black] (B) circle(0.04);
\draw[black, fill=black] (C) circle(0.04);
\draw[black, fill=black] (D) circle(0.04);
 \draw[black, fill=black] (G) circle(0.04);
 \draw[black, fill=black] (H) circle(0.04);
 \draw[black, fill=black] (M) circle(0.04);
 \draw[black, fill=black] (N) circle(0.04);
 \draw[black, fill=black] (L1) circle(0.015);
 \draw[black, fill=black] (L2) circle(0.015);
 \draw[black, fill=black] (L3) circle(0.015);
\draw[-latex] (A) -- (B);
\draw[-latex] (B) -- (C);
\draw[-latex] (C) -- (D);
\draw[-latex] (D) -- (G);
\draw[-latex] (G) -- (H);
\draw[-] (H) -- (I);
\draw[-] (J) -- (M);
\draw[-latex] (M) -- (N);
\draw[-latex] (N) -- (A);
\node[label = below :$\tiny{\{\pme^{\ba_1} , \pme^{\ba_2} \}}$] at (A) {};
\node[label = above : $\tiny{\{\pme^{\ba_1} , \pme^{\ba_2}, \pme^{\ba_3} \}}$] at (B) {};
\node[label =  below :$\tiny{\{\pme^{\bb} , \pme^{\ba_3} \}}$] at (C) {};
\end{tikzpicture}

  Therefore, in the picture above we will have $\bb \in
  \{\ba_{1}, \ba_{2}\}$, and $\{\pme^{\ba_1} ,
  \pme^{\ba_2}\}$ and $\{\pme^{\bb} , \pme^{\ba_3}\}$ are distinct
  edges of $\Trq$ with the same monomial label $$\lcm(\pme^{\ba_1} ,
  \pme^{\ba_2})= \lcm(\pme^{\bb} , \pme^{\ba_3}),$$ which
  contradicts \cref{p:edges}.

(2) is clear from the definition of $\A$.

(3) follows directly from \cref{t:BW}.

(4)  The fact that $\beta_1({\E_q}^r)$ is at most the number of
  edges of $\Srq$ follows from (3). The reverse inequality is true in
  general, as every free resolution contains the Scarf faces.

(5)  is a reformulation of (4).
\end{proof}

\section{{\bf \large  Bounds on betti numbers of powers of monomial ideals}\label{s:betti}}

In this section we discuss bounds on betti numbers that can be deduced
from our results throughout the paper. In view of \cref{t:upperbound},
knowledge of the betti numbers of $\Erq$ provides an upper bound for
the betti numbers of $I^r$ for any ideal $I$ generated by $q$
square-free monomials.

\cref{t:first-betti} shows that the first betti number of $\Erq$ equals the number of Scarf edges of $\Erq$. Thus, to explicitly find this betti number, one would like to have a combinatorial count of the Scarf edges. It turns out that a full characterization of these edges for
large values of $r$ and large $q$ is quite difficult, but small values
of these indices are manageable. We provide below a count when $r=3$.

\begin{lemma}
\label{l:types}
Let  $q\ge 2$ and $\pmea\ne \pmeb$ in ${\E_q}^3$. Then $\{\pmea, \pmeb\}\in \SS^3_q$ if and only if one of the following holds: 
\begin{enumerate}
\item  $\{\pmea,\pmeb\}=\e_u\e_v\{\e_i, \e_j\}$ with $u,v,i,j\in [q]$, $i\ne j$. 
\item $\{\pmea,\pmeb\}=\e_u\{\e_i\e_j, \e_k\e_l\}$ with $u,i,j,k,l\in [q]$, $i,j,k,l$ distinct.
\item $\{\pmea,\pmeb\}=\{\e_i\e_j\e_k, \e_u\e_v\e_w\}$ with $i,j,k,u,v,w\in [q]$ distinct. 
\item $\{\pmea,\pmeb\}=\{\e_i\e_j\e_k, \e_u^2e_v\}$ with $i,j,k,u,v\in [q]$ distinct. 
\end{enumerate}
\end{lemma}

\begin{proof}
Set  $\pmed=\egcd(\pmea, \pmeb)$ and write $\pmea=\pmed\cdot\pmeaa$ and $\pmeb=\pmed\cdot\pmebb$. By \cref{p:Scarf-face},  we have 
$$
\{\pmea, \pmeb\}\in \SS^3_q \iff \{\pmeaa, \pmebb\}\in \SS^{3-|\bd|}_{q}\,.
$$

If $|\bd|>0$, then $\SS^{3-|\bd|}_{q}=\UU^{3-|\bd|}_{q}$, in view of  \cref{c:r=2-Scarf}. 
Thus, if  $|\bd|=1$, then $\{\pmea, \pmeb\}\in \SS^3_q$ if and only (2) holds  and if $|\bd|=2$, then $\{\pmea, \pmeb\}\in \SS^3_q$ if and only (1) holds. 

Assume now $\bd=0$. It remains to show that  $\{\pmea, \pmeb\}\in \SS^3_q$ if and only if (3) or (4) hold. 

If (3) holds, then $\{\pmea, \pmeb\}\in \UU^3_q$, and $\UU^3_q\subseteq \SS^3_q$ by \cref{t:f-vector}. 

If (4) holds, then $\{\pmea, \pmeb\}\in \SS^3_q$ by \cref{r:q>4}. 

Assume now neither (3), nor (4) holds. Since $\bd=0$, we must have  either $\e_i^3\in\{\pmea, \pmeb\}$ for some $i\in [q]$ or else  $\{\pmea, \pmeb\}=\{\e_i^2\e_j, e_u^2\e_v\}$ with $i,j,u,v\in[q]$ distinct. In both cases, we observe that $|\Supp(\ba)\cup \Supp(\bb)|\le 4$, and we conclude  $\{\pmea, \pmeb\}\notin \SS^3_q$ by \cref{l:find-c}. 
\end{proof}

For the next results, we recall the convention that $\binom{u}{v}=0$ when $u<v$.

\begin{theorem}
\label{t:betti-1}
Let $q\ge 1$ and let $I$ be a monomial ideal minimally generated by
$q$ square-free monomials. Then
\begin{equation}
\label{e:betti1}
\beta_1(I^3)\le \binom{q+1}{2}\binom{q}{2}+3q\binom{q}{4}+20\binom{q}{5}+10\binom{q}{6}\,.
\end{equation}
Moreover, equality holds when $S=\SE$ and $I={\E_q}$.

\end{theorem}

\begin{proof}
We first show $\beta_1({\E_q}^3)$ is equal to the right-hand side of \eqref{e:betti1}. 

In view of \cref{t:first-betti}, $\beta_1({\E_q}^3)$ is equal to the number of edges in $\SS_q^3$. We can use thus \cref{l:types} to count the number of edges in each of the cases (1)-(4). 

The number of edges that satisfy (1) in \cref{l:types} is $\binom{q+1}{2}\binom{q}{2}$. 

The number of edges that satisfy (2) in \cref{l:types} is $3q\binom{q}{4}$.

The number of edges that satisfy (3) in \cref{l:types} is $10\binom{q}{6}$.

The number of edges that satisfy (4) in \cref{l:types} is $20\binom{q}{5}$.

Putting these numbers together, we obtain that $\beta_1({\E_q}^3)$ is equal to the right-hand side of \eqref{e:betti1}. Finally, to justify the inequality in \eqref{e:betti1}, apply \cref{t:upperbound}. 
 \end{proof}

\begin{example}
Using \cref{t:betti-1}, we see that 
$$
\beta_1({\E_4}^3)=72\qquad  \beta_1({\E_5}^3)=245 \qquad \beta_1({\E_6}^3)=715 \qquad \beta_1({\E_7}^3)=1813 \qquad \beta_1({\E_8}^3)=4088\,.
$$
(This computation can also be verified using Macaulay2. )
These numbers are thus effective bounds for $\beta_1(I^3)$ for any ideal $I$ generated by $q$ square-free monomials when $q=4,5,6,7$, respectively $8$. 
\end{example}

 \begin{theorem}[{\bf Effective bounds on betti numbers}]\label{upper bounds}
 Let $q, r\ge 1$ and let $I$ be an ideal of $S$ minimally generated by
 $q$ square-free monomials, and set $$\gamma_i=\binom{r-i+q-1}{q-1}.$$
\begin{enumerate}
\item If $q=2$, then $\pd_S(I^r)\le 1$ and 
$$
\beta_0(I^r)\le r, \qquad \beta_1(I^r)\le r-1\,.
$$
\item If $q=3$, then $\pd_S(I^r)\le 2$ and  
$$\begin{array}{lll}
  \beta_0(I^r)\le \gamma_0, &
  \beta_1^S(I^r)\le 3\gamma_1, &
  \beta_2^S(I^r)\le \gamma_1+\gamma_2. 
\end{array}$$

\item  If $q=4$, then $\pd_S(I^r) \le 5$ and 
$$\begin{array}{lll}
 \beta_0(I^r)   \le \gamma_0,        &
 \beta_1^S(I^r) \le 6 \gamma_1+3\gamma_2,  &
 \beta _2^S(I^r) \le  4\gamma_1 + 16 \gamma_2,\\
  &&\\
 \beta_3^S(I^r)\le  \gamma_1+15\gamma_2+\gamma_3,  &
 \beta_4^S(I^r)\le 6\gamma_2, &
 \beta_5(I^r)\le \gamma_2. 
\end{array}$$
\end{enumerate}
Furthermore, all inequalities above become equalities when $S=S_{\E}$ and $I=\E_q$.
 \end{theorem}
 
 \begin{proof}
In view of \cref{c:f-betti},  the $\f$-vector of the complex $\UUrq$ gives the betti numbers $\beta_i(\Erq)$ when $q\le 4$. In turn, \cref{t:upperbound} gives the desired inequalities. 

We proceed now to compute $\f(\UUrq)$. In what follows, the convention is that $\N_q^i=\emptyset$ when $i<0$. Observe that $\gamma_i=|\N_q^{r-i}|$. Clearly, when $i=0$, $\gamma_0$ is exactly the number of vertices of $\UUrq$, which is $|\Nrq|$. 
 
 Assume $q=2$. The edges of $\UU^r_q$ are of the form
 $$
 \pmea \{\e_1, \e_2\},  \qquad \ba\in \N^{r-1}_2\,.
 $$
 Thus the number of edges is $|\N^{r-1}_2|=\gamma_1$. There are no higher dimensional faces in $\UU_2^r$, and this finishes the explanation for (1). 
 
 Assume  $q=3$.  The facets are 
 $$ \pmea\{\e_j\st j\in [3]\}, \quad \ba\in \N_q^{r-1}\qand \pmea\{\e_j\e_k\st j\ne k\}, \quad \ba\in \N_q^{r-2}\,.$$
 There are $\gamma_1$ facets of the first type and $\gamma_2$ facets of the second type, hence $\gamma_1+\gamma_2$ choices in total for the $2$-dimensional faces. 
 
 To count the edges, note that they take the form below:  
 $$
 \pmea\{\e_i,\e_j \}, \quad \ba\in \N_q^{r-1},\quad  i\ne j\,.
$$
Since $|\N_q^{r-1}|=\gamma_1$ and there are $3$ choices for $i,j$ with $i\ne j$, there are $3\gamma_1$ choices in total. This finishes the explanation for (2). 

Assume $q=4$. The edges are of two types: 
 \begin{align*}
& \pmea\{\e_i,  \e_j \}, \quad \ba\in \N_q^{r-1} , \quad i\ne j\\
& \pmea\{\e_i\e_j,  \e_ke_l \}, \quad \ba\in \N_q^{r-2}  \quad \{i,j,k,l\}=[4].
\end{align*}
There are $6\gamma_1$ edges of the first type and $3\gamma_2$ edges of the second type. 

The 2-dimensional faces are of the form: 
 \begin{align*}
& \pmea\{\e_i, \e_j , \e_k \},  \quad \ba\in \N_q^{r-1}, \quad |\{i,j,k\}|=3\\
& \pmea \{\e_i\e_j, \e_ke_l ,  \e_ue_v\}, \quad \ba\in \N_q^{r-2},\quad  i\ne j, \quad  k\ne l, \quad u\ne v, \quad \{i,j\}\cap \{k,l\}\cap \{u,v\}=\emptyset.
\end{align*}
There are $4\gamma_1$ faces of the first type and $(\binom{6}{3}-4)\gamma_2=16\gamma_2$ faces of the second type. 

The $3$-dimensional faces are of the form: 
 \begin{align*}
& \pmea\{\e_1,  \e_2 ,   \e_3,    \e_4 \},\quad  \ba\in \N_q^{r-1} \\
& \pmea\{\e_i\e_j,  \e_ke_l ,   \e_ue_v,    \e_pe_q\},\,  \ba\in \N_q^{r-2} , \,  i\ne j, \, k\ne l, \, u\ne v, \, p\ne q, \,  \{i,j\}\!\cap \!\{k,l\} \!\cap \!\{u,v\} \!\cap \!\{p,q\}=\emptyset \\
& \pmea\{\e_1\e_2\e_3,  \e_1\e_2\e_4 ,   \e_2\e_3e_4\}, \quad \ba\in \N_q^{r-3} .
\end{align*}
There are $\gamma_1$ faces  of the first type, $\gamma_3$ of the third type and $\binom{6}{4}\gamma_2=15\gamma_2$ faces of the second type. 
 
The $4$-dimensional faces are of the form: 
 \begin{align*}
 \pmea\{\e_i\e_j,  \e_k\e_l ,   \e_u\e_v,    \e_p\e_q,  \e_b\e_c\}, \, \ba\in \N_q^{r-2} ,&\, i\ne j, \, k\ne l, \, u\ne v, \, p\ne q, \, b\ne c\,,\\
& \, \{i,j\}\cap \{k,l\}\cap \{u,v\}\cap \{p,q\}\cap \{b,c\}=\emptyset\,.
\end{align*}
There are $6\gamma_2$ faces of this form. 

Finally, the $5$-dimensional faces are of the type
$$\pmea\{\e_1\e_2,  \e_1\e_3 ,   \e_1\e_4,    \e_2e_3,  \e_2\e_4 ,   \e_3\e_4 \},  \quad \ba\in \N_q^{r-2} $$
hence there are $\gamma_2$ of them. 
 \end{proof}

For the sake of completeness, we also include below the bounds obtained in \cite{L2} in the case when $r=2$. The fact that the bounds are achieved when $I=\E_q$ is established in \cite{Lr}. 

\begin{theorem}[{\cite[Theorem 4.1]{L2},\cite[Proposition 7.9]{Lr}}]
\label{t:L2r}
Let $q\ge 1$ and let $I$ be a monomial ideal minimally generated by
$q$ square-free monomials. Then
$$
\beta_i(I^2)\le{{\frac{1}{2}(q^2-q)}\choose{i+1}}+q {{q-1}\choose{i}}\,.
$$
Furthermore, equality holds when $S=S_{\E}$ and $I=\E_q$.
\end{theorem}

While we succeeded in providing effective bounds on the betti numbers
of powers of square-free monomials in certain cases, as recorded
above, it remains to be established whether such bounds can be given
for all $r$ and $q$. In particular, this paper raises the question:

\begin{question}
\label{q1}
Is the minimal free resolution of $\Erq$ supported on a simplicial (or cellular) complex for all $q\ge 1$ and all $r\ge 1$? In particular, are the ideals $\Erq$ Scarf?
\end{question}

If this is true, then one would also want to explicitly describe the
faces of the simplicial (or cellular) complex in \cref{q1}, with the
hope that one can obtain explicit formulas for bounds on betti
numbers, as given in this section. While we were able to provide a
clear description of $\Srq$ when $q\le 4$ and also when $r\le 2$, such a
description seems to become a lot more complex for larger $q$ or
$r$. We believe one can show that the ideals $\Srq$ are Scarf at least
when $q=5$ or when $r=3$ as well, but the difficulty of understanding
the Scarf complex with our current techniques is increasing sharply
for large values of $q$ and $r$.

\subsubsection*{Acknowledgements} The research for this paper took
place in May 2022, during a two-week stay of the authors at the
Mathematisches Forschungsinstitut in Oberwolfach (MFO), as Oberwolfach
Research Fellows. We are grateful to MFO for providing us with a
wonderful environment and facilities to do concentrated work.

We are grateful to Susan Cooper and Susan Morey with whom we have an
on-going collaboration on products of simplicial complexes. Their
insights will have inevitably influenced us while writing the current
paper.
  
Sara Faridi's research is supported by the Natural Sciences and
Engineering Research Council of Canada (NSERC) Discovery Grant program
(\#2023-05929). Liana \c{S}ega was supported in part by a grant from
the Simons Foundation (\#354594). For Sandra Spiroff, this material is
based upon work supported by and while serving at the National Science
Foundation. Any opinion, findings, and conclusions or recommendations
expressed in this material are those of the authors and do not
necessarily reflect the views of the National Science Foundation.

The computations for this paper were made using the computer algebra
softeware Macaulay2~\cite{M2}.


\end{document}